\documentclass[12pt]{amsart}

\usepackage{latexsym,amsmath,amssymb,amsthm,amsfonts}

\usepackage[shortalphabetic]{amsrefs}

\usepackage{hyperref}

\numberwithin{equation}{section}

\usepackage[capitalise]{cleveref}

\usepackage{multido,pst-plot,pstricks,pst-node}
\newcommand{\R}{{\mathbb{R}}}
\newcommand{\eps}{\varepsilon}

\renewcommand{\P}{{\Pi}}

\newcommand{\T}{{\mathcal{T}}}
\renewcommand{\S}{{\mathbb{S}}}
\newcommand{\x}{{\boldsymbol{x}}}

\newcommand{\bxi}{{\boldsymbol{\xi}}}

\renewcommand{\u}{{\boldsymbol{u}}}

\newcommand{\w}{{\boldsymbol{w}}}

\newcommand{\y}{{\boldsymbol{y}}}

\renewcommand{\b}{{\boldsymbol{b}}}
\newcommand{\z}{{\boldsymbol{z}}}

\newcommand{\wt}{\widetilde}

\newcommand{\bg}{{\boldsymbol{g}}}
\newcommand{\bu}{{\boldsymbol{u}}}
\newcommand{\bv}{{\boldsymbol{v}}}
\newcommand{\bde}{{\boldsymbol{e}}}

\newcommand{\f}{\frac}
\newcommand{\Og}{\Omega}
\newcommand{\Bl}{\Bigl}
\newcommand{\Br}{\Bigr}
 \def\sph{\mathbb{S}^{d-1}}
 \def\ta{\theta}
 \def\og{\omega}

\newcommand{\bta}{{\boldsymbol{\eta}}}
\def\bog {\boldsymbol{\omega} }
\renewcommand{\phi}{{\boldsymbol{\varphi}}}
\def\al{\alpha}
\def\b{\beta}
\def\Ld{\Lambda}
\def\da{\delta}
\def\p{\partial}
\def\0{\boldsymbol{0}}
\def\SS{\mathbb{S}}
\def\u{\bu}
\def\bN{\boldsymbol N}
\def\w{\boldsymbol w}
\def\RR{\mathbb{R}}
\def\ld{\lambda}

\newcommand{\qtq}[1]{{\quad\text{#1}\quad}}

\newtheorem{theorem}{Theorem}[section]
\newtheorem{lemma}[theorem]{Lemma}
\newtheorem{proposition}[theorem]{Proposition}

\theoremstyle{remark}
\newtheorem{remark}[theorem]{Remark}

\title{Optimal polynomial meshes exist on any multivariate convex domain}

\author{Feng Dai}
\address{Department of Mathematical and Statistical Sciences,
	University of Alberta, Edmonton, Alberta T6G 2G1, Canada}
\email{fdai@ualberta.ca}

\author{Andriy Prymak}
\address{Department of Mathematics, University of Manitoba, Winnipeg, MB, R3T 2N2, Canada}
\email{prymak@gmail.com}

\thanks{	The first author was supported by  NSERC of Canada Discovery
	grant RGPIN-2020-03909, and the second author  was supported by NSERC of Canada Discovery grant RGPIN-2020-05357.
}



\begin{document}

\begin{abstract}
We show that optimal polynomial meshes exist for every convex body  in $\R^d$,  confirming a conjecture by A.~Kro\'{o}.
\end{abstract}

\maketitle

\section{Introduction}\label{sec:intr}

  For  a  compact set $E\subset \R^d$ and a continuous function $f$ on $E$, we define $\|f\|_{E}:=\max_{\x\in E}|f(\x)|$.  We denote by $\Pi_{n}^d$ the space of all real algebraic polynomials in $d$ variables of total degree at most $n$.  A compact domain    $\Omega$   in $\R^d$ is said to   possess {\it optimal polynomial meshes} if  there exists a sequence $\{Y_n\}_{n\ge1}$ of finite subsets of $\Omega$ such that the cardinality of $Y_n$ is at most $C_1 n^d$ while
\begin{equation}\label{eqn:infty norm discr}
	\|P\|_{\Omega}\le C_2 \|P\|_{Y_n} \qtq{for any} P\in\P_{n}^d,
\end{equation}
where $C_1$ and  $C_2$ are positive   constants depending only on $\Omega$.  Note that the dimension of the space $\P_{n}^d$ is of the order $n^d$ as $n\to\infty$, which is the reason for calling such sets {\it optimal} meshes. Results on existence of optimal meshes can be viewed as Marcienkiewicz-Zygmund inequalities for discretization of $L_\infty$ norm, see~\cite{DPTT}. Optimal meshes (being partial case of the so-called \emph{norming sets}, i.e. sets satisfying~\eqref{eqn:infty norm discr}) have applications for discrete least squares approximation, cubature formulas, scattered data interpolation, study of discrete Fekete and Leja type sets, see for example~\cite{ap1,ap2,ap3}.

Existence of optimal polynomial meshes was previously established for various classes of domains in $\R^d$, including convex polytopes in~\cite{Kr11},  $C^\alpha$ star-like domains with $\alpha>2-\frac2d$ in~\cite{Kr13}, and  certain extension of $C^2$ domains in~\cite{Pi}.  It was conjectured by Kro\'{o}~\cite{Kr11} that each  convex compact set with nonempty interior  possesses optimal polynomial meshes.  In~\cite{Kr19} Kro\'{o} settled this conjecture for $d=2$ proving existence of optimal polynomial meshes for arbitrary planar convex domains using certain tangential Bernstein inequality.  The  second author of the current paper gave an alternative proof  of this  conjecture for $d=2$,    using a connection between the Christoffel functions, positive quadrature formulas and polynomial meshes established in the  paper~\cite{Bo-Vi} of  Bos and Vianello. We also mention that for any compact set $\Omega$ in $\R^d$ existence of \emph{nearly} optimal meshes, i.e. those  satisfying~\eqref{eqn:infty norm discr} with cardinality of $Y_n$ at most $C (n\log n)^d$, was established in~\cite{BBCL}*{Prop.~23} using Fekete points which are extremely hard to find explicitly.

%[[should we also refer to the recent results on MT type discretization?]]

The main goal of this paper is to prove  the following theorem, which confirms the   conjecture of Kro\'{o}  for $d\ge 3$.
\begin{theorem}\label{thm:optimal mesh}
	There exists a constant $C$ depending only on $d$ such that for every positive integer $n$ and  every convex body $\Og\subset \R^d$, there exist   $\x_1,\dots, \x_N\in\Og$ with  $N\leq C n^d$   such that
	\begin{equation}\label{1-2:eq}\|Q\|_\Og \leq 2 \max_{1\leq j\leq N}|Q(\x_j)|\  \ \text{ for every $Q\in \Pi_n^d$.}\end{equation}
\end{theorem}
\noindent It is worthwhile to point out that the constant $C$ in this  theorem  does not depend  on the particular geometry of the convex body $\Omega$.

We rely on certain ideas from the paper~\cite{Du} of Dubiner, but the proofs here will be self-contained and independent of~\cite{Du}. One of the key concepts in~\cite{Du} is a metric on a convex body $\Omega\subset\R^d$ defined as follows. For any $\x,\y\in\Omega$, define
\begin{equation}\label{eqn:def-dubiner}
	\rho(\x,\y):=\rho_\Omega(\x,\y):=\max_{\bxi\in\S^{d-1}}
	\left|\sqrt{\x\cdot\bxi-a_\bxi}-\sqrt{\y\cdot\bxi-a_\bxi}\right|,
\end{equation}
where $a_\bxi:=\min_{\z\in\Omega} \z\cdot \bxi$, and $\S^{d-1}$ denotes the unit sphere of $\R^d$. It is straightforward to verify that $\rho$ is a metric on $\Omega$. The desired optimal mesh will be constructed as a set of points satisfying certain separation and covering conditions with respect to (w.r.t.) the  metric $\rho$.

%[[on constructivity of the mesh by the proof]]

In the next section, we briefly explain some ideas behind
the proof of Theorem \ref{thm:optimal mesh}, describe  several intermediate  results that are  required in the proof and may be of independent interest, as well as give the structure of the remainder of this paper. 

\section{Outline of the proof}\label{sec:2}

The proof of Theorem \ref{thm:optimal mesh} is long and may appear quite technical, but, in fact, it follows certain rather geometric ideas which will be described below. 

We use $\|\cdot\|$ to denote the Euclidean norm of a vector in $\R^d$, $\0$ for the origin in $\R^d$, and denote by $B(\x, r):=B(\x, r)_{\R^{d}}:=\{\z\in \R^d:\  \ \|\z-\x\|\leq r\}$ the closed  Euclidean ball with center $\x\in\R^d$ and radius $r>0$. We use $\lambda_d$ to denote the Lebesgue measure on $\R^d$.

We start with a few useful observations. % that will be needed in the proof of Theorem \ref{thm:optimal mesh}.
Due to John's theorem on inscribed ellipsoid of the largest volume\footnote{There appears to be no consistency in the literature for the names of the maximal volume inscribed ellipsoid and the minimal volume circumscribed ellipsoid; either one may be referred to as John's or L\"owner-John's or L\"owner's ellipsoid. According to Busemann, L\"owner discovered the uniqueness of the minimal volume ellipsoid, but this was never published. John established a characterization for these ellipsoids which implied uniqueness and other properties. An interested reader is referred to the survey~\cite{He}.}~\cite{Sc}*{Th.~10.12.2, p.~588}, for any convex body $\Omega\subset\R^d$,  there exists a nonsingular  affine transform $\T:\R^d\to\R^d$ such that $
 B(\0,1)\subset \T \Omega\subset B(\0,d)$.  Since optimal meshes are invariant under affine transforms of the domain, without loss of generality we may assume that \begin{equation}\label{2-1}
 B(\0,1)\subset \Omega\subset B(\0,d),
 \end{equation} which will  allow us to achieve that all involved implicit  constants depend on $d$ only, but not on the particular geometry of the boundary of $\Omega$.
 % We will keep this assumption throughout the proof (unless otherwise stated).

 While optimal meshes are invariant  under affine transforms of the domain,   the Dubiner distance $\rho_{\Og}$ is clearly  not.  However, it can be shown (see~\eqref{eqn:metric after affine transform}) that the metric $\rho_\Og$ is equivalent (up to some constants) under affine transforms.  Let $\rho=\rho_{\Og}$ denote the metric given in \eqref{eqn:def-dubiner}.  For $\x\in\Og$ and $r>0$, we define  \begin{equation}\label{2-7}
B_\rho(\x,r):=\{\y\in\Omega: \rho(\x,\y)\le r\}.
\end{equation}
Given a number $\epsilon>0$, we say a finite subset  $\Ld$ of $ \Og$ is    $\epsilon$-separated w.r.t.  the metric $\rho$   if $\rho (\og, \og')\ge \epsilon$ for any two distinct $\og, \og'\in\Ld$. An $\epsilon$-separated subset $\Ld$ of $\Og$ is called maximal if
$\Og=\bigcup_{\og\in\Ld} B_\rho (\og, \epsilon).$ 
We remark that for any $\epsilon>0$, there exists a finite maximal $\epsilon$-separated subset of $\Omega$, moreover, this subset can be obtained in a constructive manner. Namely, choose arbitrary $\x_1\in\Omega$ and proceed recursively. If $\{\x_1,\dots,\x_n\}$ is an $\epsilon$-separated subset of $\Omega$ which is not maximal, then define $\x_{n+1}$ to be any point from the nonempty set $\Omega\setminus \left(\bigcup_{j=1}^n B_\rho (\x_j, \epsilon)\right)$ in which case $\{\x_1,\dots,\x_{n+1}\}$ is a larger $\epsilon$-separated set. It remains to show that this process must terminate. Indeed, otherwise, since $\Omega$ is a compact set (w.r.t. the Euclidean metric), there exists a convergent subsequence $\{\x_{n_k}\}_{k=1}^\infty$. Using an elementary inequality $|\sqrt{t}-\sqrt{s}|\le\sqrt{|t-s|}$, $t,s\ge 0$, we see that $\rho(\x,\y)\le \sqrt{\|\x-\y\|}$ for any $\x,\y\in\Omega$. This gives a contradiction by $\epsilon \le \rho(\x_{n_k},\x_{n_{k+1}})\le \sqrt{\|\x_{n_k}-\x_{n_{k+1}}\|}\to0$, $k\to\infty$.

%Indeed, we can start with an arbitrary point $\Omega$ and construct a sequence $Y_n$ of finite $\epsilon$-separated subsets of $\Omega$. If $Y_n\subset\Omega$ is not maximal, then pick an arbitrary point $\x_{n+1}$ from the nonempty set $\Omega\setminus \left(\bigcup_{\og\in Y_n} B_\rho (\og, \epsilon)\right)$, and define $Y_{n+1}:=Y_n\cup\{\x_{n+1}\}$ to obtain a bigger $\epsilon$-separated subset. It remains to show that this process must terminate. Otherwise, % which requires a property of $\rho$ which will be established in~\cref{thm:doubling}. Let $k$ be a positive integer satisfying $2^k\epsilon/3\ge \sqrt{D}$, where $D$ is the (Euclidean) diameter of $\Omega$. We observe that for any $\x\in\Omega$ we have $\Omega=B_\rho(\x,\sqrt{D})=B_\rho(\x,2^k\epsilon/3)$. For every $n$, the sets $B_\rho(\x,\epsilon/3)$, $\x\in Y_n$ are disjoint, while by iterating~\cref{thm:doubling} $k$ times we have $4^{kd}\lambda_d(B_\rho(\x,\epsilon/3))\ge\lambda_d(\Omega)$, implying $n\le 4^{kd}$.

Now let $Y_n\subset\Omega$ be a maximal $c_d/n$-separated set w.r.t. the metric $\rho$, where $c_d\in (0, 1)$ is a  small constant depending only on $d$.  The proof of Theorem \ref{thm:optimal mesh} consists of two key  components:   \textbf{i)} show that the estimate ~\eqref{1-2:eq} holds provided that the constant $c_d$ is small enough;  and \textbf{ii)} estimate   the cardinality of $Y_n$.

For the component \textbf{i)},  we obtain  the estimate \eqref{1-2:eq} as   an immediate  consequence of the following theorem establishing a pointwise bound on how the polynomials can vary in terms of the distance $\rho$. 
\begin{theorem}\label{thm:norm control} Let $\Omega\subset\R^d$ be a convex body satisfying~\eqref{2-1}.
There exists a positive  constant  $C_\ast$ depending only on $d$  such that  for any $Q\in\P_{n}^d$ with $\|Q\|_{\Omega}\le1$,  we have
	\begin{equation}\label{eqn:pxy control}
		|Q(\x)-Q(\y)|\le  C_\ast n\rho(\x,\y) \qtq{whenever}  \x,\y\in\Omega.
	\end{equation}
\end{theorem}
Theorem~\ref{thm:norm control} will be proved in Section~\ref{sec 4:lec}.
The proof of Theorem \ref{thm:norm control} is essentially two-dimensional which is due to the fact that any supporting line to a 2-dimensional section of $\Omega$ can be extended to a supporting hyperplane for $\Omega$. The non-trivial case is when $\x$ and $\y$ are both close to the boundary of $\Omega$. In this case, we consider the planar section $G$ of $\Omega$ through the origin and the points $\x$ and $\y$. After an appropriate affine transform, the boundary of $G$ can be effectively parametrized so that the main geometric ingredient of  Lemma~2.4 of~\cite{Pr} is applicable. This allows us  to explicitly construct a parallelogram in $G$ containing both $\x$ and $\y$, as well as certain families of parabolas and straight segments in $G$ along which the standard one-dimensional Bernstein inequality can be applied to derive~\eqref{eqn:pxy control}. We find that it suffices to use  $\rho(\x,\y)$    with  the maximum in \eqref{eqn:def-dubiner}  being taken  over only two specific directions  which are inward normal  vectors at two boundary  points of  $\Omega$. This leads to  a certain related two-dimensional version of the metric $\rho$ which is easier to compute.

For the component \textbf{ii)},  it is sufficient to show that  the cardinality of $Y_n$ does not exceed the dimension  of $\P_{\alpha n}^d$ for a sufficiently large positive  constant  $\alpha$ depending only on $d$. Indeed, if this were not true, then one can get a contradiction using linear dependence of the  ``fast decreasing'' polynomials provided by the following theorem for  each $\x\in Y_n$. We would like to mention that the study of fast decreasing polynomials has been originated in~\cite{IvTo}, with multivariate polynomials of ``radial'' structure covered in~\cite{Kr16}.
\begin{theorem}\label{thm:polynomial construction}  Let $\Omega\subset\R^d$ be a convex body satisfying~\eqref{2-1}.
	For any $\x\in \Omega$ and $n\in\mathbb{N}$,  there exists a polynomial  $P\in\P_{ n}^d$ such that $P(\x)=1$ and
	\begin{equation}\label{eqn:fast decreasing}
		0\le P(\z)\le  C \exp(-c\sqrt{n\rho(\x,\z)})\   \qtq{for any} \z\in\Omega,
	\end{equation}
	where $C>1$  and $c\in (0, 1)$ are  constants depending only on $d$.
\end{theorem}

Theorem \ref{thm:polynomial construction} is proved in Section \ref{sec6:lec}.
An important ingredient for  the  proof of \cref{thm:polynomial construction} and for obtaining the above mentioned  contradiction is the doubling property  stated in the following theorem which is proved  in Section~\ref{sec5:lec}.
\begin{theorem}\label{thm:doubling}  Let $\Omega\subset\R^d$ be a convex body satisfying~\eqref{2-1}.
	For any $\x\in\Omega$ and $h>0$
	\begin{equation}\label{eqn:doubling}
		\lambda_d(B_\rho(\x,2h))\le 4^d \lambda_d(B_\rho(\x,h)).
	\end{equation}
\end{theorem}
Finally, the proof of our main result,  Theorem \ref{thm:optimal mesh}, is given in Section \ref{sec 7:lec}.

We conclude this section with remarks about two extensions of the main result. We chose not to include them into Theorem~\ref{thm:optimal mesh} for clarity of the statement. 
\begin{remark}
	A minor modification of the proof yields the following $\eps$-version of Theorem~\ref{thm:optimal mesh}. For any $\eps>0$ there exists a constant $C$ depending only on $d$ and $\eps$ such that for every positive integer $n$ and  every convex body $\Og\subset \R^d$, there exist   $\x_1,\dots, \x_N\in\Og$ with  $N\leq C n^d$   such that
	\begin{equation}\label{eqn:eps-version}\|Q\|_\Og \leq (1+\eps) \max_{1\leq j\leq N}|Q(\x_j)|\  \ \text{ for every $Q\in \Pi_n^d$.}\end{equation}
\end{remark}

\begin{remark}
	A straightforward change in the proof implies that for any $\eps>0$ there exists a constant $c_d>0$ depending only on $d$ and $\eps$ such that any $\Lambda:=\{\x_1,\dots, \x_N\}\subset\Omega$ which is $c_d/n$-covering w.r.t. $\rho$, i.e. $\Og=\bigcup_{j=1}^N B_\rho (\x_j, c_d/n)$, satisfies~\eqref{eqn:eps-version}. Furthermore, if $\Lambda$ is an $\eta/n$-separated set w.r.t. $\rho$, where $0<\eta\le c_d$ is an absolute constant, then $N\le C(\eta,d,\eps) n^d$. In particular, this allows more flexibility for the construction of optimal meshes as a suitable $\eta/n$-separated set does not need to be maximal, but instead is required to be $c_d/n$-covering with $c_d$ being larger than $\eta$ (for a maximal set we have $\eta=c_d$).
\end{remark}

\section{Variation of polynomials}\label{sec 4:lec}

This section is devoted to the proof of \cref{thm:norm control}.
%Throughout the proof, we use the letters $C_1$,  $c_1$ to denote a sufficient large and a sufficiently small positive constants respectively, both depending only on $d$.  The  exact values of $C_1$ and $c_1$ are not important to us.

Let  $Q\in \Pi_n^d$  be such that $\|Q\|_{\Og}=1$. Our aim is to show that
\begin{equation}\label{3-2}
|Q(\x)-Q(\y)|\leq C_\ast n\rho(\x,\y),\  \ \x, \y\in\Og.
\end{equation}
 Without loss of generality, we may assume that $0<\rho(\x,\y)\leq \f {c_1}n$ (in particular, we can assume that $n$ is sufficiently large),  since otherwise \eqref{3-2} with $C_\ast=\f 2{c_1}$ holds trivially. Here and throughout this section, $c_1$ denotes a fixed  sufficiently small positive constant depending only on $d$.
 By symmetry, we can  also assume that $\x\neq \0$.  Then  by convexity and  \eqref{2-1},   there exists a unique number $\da_{\Og}(\x)\in [0, d)$ such that $\x^\ast =\x+\da_\Og(\x) \f {\x}{\|\x\|}\in\p\Og$. Geometrically, $\delta_\Omega(\x)$ is the distance from $\x$ to $\partial\Omega$ along the ray with the initial point $\0$ in the direction of $\x$.

The proof of \eqref{3-2} is quite long, so we break it into several parts. First, in Subsection \ref{sec:3-1}, we collect some  known facts and results on  convex functions and algebraic polynomials in one variable,   which  will be needed in the proof of \eqref{3-2}. After that, in Subsection \ref{subsec4-1}, we treat  the simpler  case when   $\da_{\Og}(\x) \ge \f 1{100}$ or  $0< \da_{\Og} (\x)\leq \f {C_1^2}{n^2}$. Here and throughout this section, the letter $C_1$ denotes a fixed  sufficiently large  positive constant depending only on $d$.
Indeed, if $\da_{\Og}(\x) \ge \f 1{100}$, then both  $\x$ and $\y$ lie ``far'' away from  the boundary $\p\Og$ of $\Og$, and \eqref{3-2} can be deduced directly from the one-dimensional Bernstein inequality.
On the other hand,
 if   $0< \da_{\Og} (\x)\leq \f {C_1^2}{n^2}$, then one can use  Remez's  inequality to reduce  the problem to the case of $ \da_{\Og} (\x)\ge  \f {C_1^2}{n^2}$ for a slightly wider class of domains $\Og$ satisfying
 \begin{equation}\label{eqn:2-1relaxed}
 	B(\0,1)\subset\Omega\subset B(\0,2d)
 \end{equation}
instead of~\eqref{2-1}.
Finally,  Subsections  \ref{subsec4-2} and \ref{subsec4-3} are devoted to the proof of \eqref{3-2} for  the  remaining more involved case of  $\f {C_1^2}{n^2}\leq \da_{\Og} (\x)\leq  \f 1 {100}$ under the assumption~\eqref{eqn:2-1relaxed}.
Indeed,   in Subsection  ~\ref{subsec4-2},  we show that the problem \eqref{3-2}  can be reduced to an inequality   related to  a two-dimensional version of the metric $\rho$    on a planar convex domain. Such a two-dimensional   inequality  is  proved
in Subsection ~\ref{subsec4-3}.

\subsection{Notations and preliminaries} \label{sec:3-1}
 %We will  use boldface letters $\x,\y,\z,\dots$ to denote points in $\R^d$. 
Let $\bde_1=(1,0,\dots,0)$, $\dots$, $\bde_d=(0,\dots, 0, 1)$ denote the standard canonical basis of $\R^d$. We call the coordinate axis in the direction of $\bde_i$ the $i$-th coordinate axis  for  each $1\leq i\leq d$. 
For a finite set $E\subset \RR^d$, we denote by % $|E|$  the  Lebesgue measure of  $E$ in  $\RR^d$, and 
$\# E$ the cardinality of $E$. %Finally, we denote by   $[\x, \y]$  the line segment connecting any two points $\x, \y\in\RR^{d}$; that is, $[\x, \y]=\{ t\x+(1-t) \y:\   \  t\in [0,1]\}$.  
The boundary and the interior of a subset $K$ of $\R^d$ w.r.t. the Euclidean metric are denoted as $\partial K$ and $K^\circ$, respectively. For $K\subset \R^d$, we let ${\rm conv}\,(K)$ be the convex hull of $K$.

Now let us discuss certain properties of the metric~\eqref{eqn:def-dubiner}. First, we will establish equivalence under affine transforms. If $\T$ is a non-degenerate affine transform on $\R^d$,
then for each  fixed $\bxi\in\S^{d-1}$ and any $\x,\y\in\Omega$,  we have
\begin{align*}
	\sqrt{\T \x \cdot \bxi - \min_{\z\in\Omega} \T \z \cdot \bxi}
	- &
	\sqrt{\T \y \cdot \bxi - \min_{\z\in\Omega} \T \z \cdot \bxi} \\
	&\quad=\sqrt{ \x \cdot \T^*\bxi - \min_{\z\in\Omega}  \z \cdot \T^*\bxi}
	-
	\sqrt{ \y \cdot \T^* \bxi - \min_{\z\in\Omega} \z \cdot \T^* \bxi} \\
	&\quad =\|\T^*\bxi\|^{\frac12} \left(
	\sqrt{\x \cdot \bta - \min_{\z\in\Omega} \z \cdot \bta}
	-
	\sqrt{\y \cdot \bta - \min_{\z\in\Omega} \z \cdot \bta}
	\right),
\end{align*}
where $\bta:=\frac{\T^*\bxi}{\|\T^*\bxi\|}\in\S^{d-1}$,  and $\T^*$ denotes  the transpose of the linear mapping $\T -\T (\0)$. Therefore, taking the maximum over $\bxi\in\sph$ yields
\begin{equation}\label{eqn:metric after affine transform}
	\rho_{\T\Omega}(\T\x,\T\y)\le \|\T^*\|^{\frac12}\, \rho_{\Omega}(\x,\y),
\end{equation}
where
$\|\T^\ast\|:= \max_{\bxi\in\S^{d-1}}\|\T^*\bxi\|$.
Clearly, we can also apply this last estimate to the inverse  of  $\T$, leading to  the equivalence
\begin{equation}\label{eqn:metric after affine transform}
	\|\T^\ast \|^{-\f12} \rho_{\T\Omega}(\T\x,\T\y)\le  \rho_{\Omega}(\x,\y)\leq \|(\T^\ast)^{-1} \|^{\f12} \rho_{\T\Omega}(\T\x,\T\y),\   \   \forall \x,\y\in\Og.
\end{equation}
The metric $\rho_{\Og}$  possesses several other useful  properties as well. First,
by \eqref{eqn:def-dubiner},
\begin{align}
	\rho_\Omega(\x,\y)&=\max_{\bxi\in\S^{d-1}}\f { |(\x-\y)\cdot \bxi|}
	{\sqrt{\x\cdot\bxi-a_\bxi}+\sqrt{\y\cdot\bxi-a_\bxi}}\ge \max_{\bxi\in\S^{d-1}}\f { |(\x-\y)\cdot \bxi|}
	{2\max_{\z, \z'\in\Og} \sqrt{(\z-\z')\cdot\bxi}}.\label{2-4}
\end{align}
Under the assumption \eqref{2-1}, we have
\[ \max_{\z, \z'\in\Og} (\z-\z')\cdot\bxi\leq  \max_{\z\in B(\0, 2d)} \z\cdot \bxi\leq 2d,\   \  \forall \bxi\in\sph,\]
hence
\begin{equation}\label{2-5}
	\|\x-\y\|\leq 2\sqrt{2d}\,  \rho_\Og(\x,\y),\   \    \forall \x,\y\in\Og.
\end{equation}
Second, by the definition and \eqref{2-4}, it is clear that if $\wt\Og$ is a convex set  such that  $\Og\subset \wt \Og$, then
\begin{equation}\label{2-6}
	\rho_{\wt \Og} (\x,\y) \leq \rho_{\Og} (\x, \y),\   \   \ \forall \x,\y\in\Og.
\end{equation}

Next we need some background on convex functions. Let $f: D \rightarrow \mathbb{R}$ be a continuous and convex function  on a convex domain $D\subset \R^d$. A vector $\bg$ in $\mathbb{R}^{d}$ is called a subgradient of $f$ at a point $\bv \in D$ if
$$
f(\bu)-f(\bv)- \bg\cdot (\bu-\bv) \geqslant 0 \quad \forall \bu \in D .
$$
Denote by $\partial f(\bv)$  the set of all subgradients of $f$ at $\bv\in D$.
As is well known,   the set  $\partial f(\bv)$  is nonempty, compact and convex for each $ \bv\in D^\circ$.
The directional derivative of $f$ at a point $\bu \in D^\circ$ in the direction $\bxi\in\R^d$ is defined by
$$
D_{\bxi} f(\bu) := \lim _{t \to 0+ } \frac{f(\bu+t \bxi )-f(\bu )}{t} .
$$
This quantity always exists and is finite  since $f$ is convex and continuous.
Moreover, we have
\begin{equation}\label{6-1}
D_{\bxi} f(\bu) =\sup _{\bg \in \partial f(\bu )} \bg\cdot  \bxi,\   \ \bxi\in\R^d.
\end{equation}
Subgradients possess several properties.
First, if $\al\ge 0$, then $\p (\al f)(\u)=\al \p f(\u)$.
Second, if $f=f_1+f_2+\dots+f_m$ and each $f_j$ is convex on $D$, then
\[ \p f(\u)=\p f_1(\u)+\dots+\p f_m (\u),\   \    \u\in D.\]
Third, if  $f$ is convex on $\R^d$, $A$ is a $d\times d$ nonsingular matrix, $\x_0\in\R^d$ and $h(\x)=f(A \x +\x_0)$, then
\[ \p h(\x) =A^t (\p f)(A\x+\x_0),\  \ \x\in\R^d.\]
In the case of one variable, for each convex function $f:[a,b]\to\R$,  the one-sided derivatives $f_-'$ and $f'_+$ exist and  are non-decreasing on $(a,b)$, with $f'_-(x)\le f'_+(x)$ for all  $x\in(a,b)$.  For further results  on  convex functions, we  refer to~\cite{BV, Ro}.

Finally, the following one-dimensional Bernstein inequality for algebraic polynomials, which will be used  repeatedly  in this section, can be found in  the book \cite[p.98]{De-Lo}.
\begin{lemma}\cite[p.98]{De-Lo}
	Let $-\infty <a<b<\infty$. Then for any $P\in\Pi_n^1$ and $x\in [a,b]$,
	\begin{equation}\label{1D-Bernstein}
	|P'(x)| \sqrt{(x-a)(b-x)} \leq n \|P\|_{[a,b]}.
	\end{equation}
\end{lemma}

\subsection{Reduction to the case of   $ \f {C_1^2}{n^2}\leq \da_{\Og}(\x) \leq  \f 1{100}$ }\label{subsec4-1}
In this subsection, we prove that
 it is sufficient to show  \eqref{3-2} for  the case of  $\f {C_1^2}{n^2}\leq \da_{\Og} (\x)\leq  \f 1 {100}$, where   $\x\in\Og\setminus\{0\}$, $ \y\in\Omega$ are   two fixed points satisfying
 $n\rho(\x,\y) \leq c_1$.

	To see this, we first assume that      $\da_{\Og}(\x) \ge \f 1{100}$.   Since $\|\x^\ast\|\leq d$ and $B:=B(\0,1)\subset \Og$, it is easily seen  that  $$B(\x, \da_\Og(\x)/d) \subset \text{conv} \Bl\{ \x^\ast\cup B\Br\}\subset \Og.$$
On the other hand, using \eqref{2-5}, we have
 $$\|\x-\y\|\leq 2\sqrt{2d}\rho(\x, \y)\leq \f{2\sqrt{2d} c_1}n<r_0:=\f 1 {200 d}$$
 provided that the constant  $c_1$ is small enough. Thus,
 $$\y\in B(\x, r_0)\subset B(\x, 2r_0)\subset \Og.$$
 Now setting $\bxi:=\f {\y-\x}{\|\y-\x\|}$, and applying the Bernstein inequality \eqref{1D-Bernstein}  along the line segment $[\x-2r_0\bxi, \x+2r_0\bxi]\subset \Og$, we obtain
\begin{align*}
|Q(\x)-Q(\y)| &\leq \|\x-\y\| \max_{0\leq t \leq  r_0}\Bl| \f {d}{dt} Q(\x+t\bxi)\Br| \leq  2\sqrt{2d} \rho(\x,\y) \frac{n}{\sqrt{3}r_0}  \|Q\|_{[\x-2r_0\bxi, \x+2r_0\bxi]} \\
&\leq  C_d n \rho(\x,\y),
\end{align*}
proving the inequality \eqref{3-2}  for the  case of  $\da_{\Og}(\x) \ge \f 1{100}$.

Next, we show  that for the proof of  \eqref{3-2} for  $0<\da_\Og(\x) \leq  \f 1{100}$, it is enough to consider  the case  of   $\da_\Og(\x) \ge \f {C_1^2}{n^2}$.    To see this, let $\ld:=1+4n^{-2} C_1^2$, and
   define $ \wt \Og=\ld\Og$.
   Clearly,   $\wt \Og$ is a convex domain  satisfying
   $ \Og\subset \wt\Og \subset B\Bl(0,  \ld d\Br)$.
  Moreover,
   \[\da_{\wt\Og} (\x)=\f {4 C_1^2} {n^2} \|\x\|+\da_\Og(\x) \ld\ge \f {C_1^2}{n^2},\]
   where the last step uses the fact that $\|\x\|\ge 1-\da_{\Og}(\x)>\f12$.
   On the other hand, setting
    $$h_\Og(\bxi): =\max\Bl\{ t\in [0,d]:\  \  t \bxi \in\Og\Br\},\  \  \bxi\in\sph,$$ and using  the univariate Remez inequality (see \cite{MT2}), we obtain
   \begin{align*}
   \|Q\|_{\wt \Og} &=\max_{\bxi\in\sph} \max _{0\leq r \leq \ld h_{\Og}(\bxi)} |Q(r\bxi) | \leq \max_{\bxi\in\sph} \max _{0\leq r \leq h_{\Og}(\bxi)  +4C_1^2 d n^{-2}} |Q(r\bxi) | \\
   &\leq C_d \max_{\bxi\in\sph} \max _{0\leq r \leq h_{\Og}(\bxi) } |Q(r\bxi) |=C_d \|Q\|_{\Og}.
   \end{align*}
   Thus, once  \eqref{3-2} is proved on the    convex body $\wt\Og$  for   the case when   $\da_{\wt \Og}(\x) \ge \f {C_1^2}{n^2}$, then using \eqref{2-6}, one has
   \[ |Q(\x)-Q(\y)|\leq C n \rho_{\wt\Og}(\x,\y) \leq C n  \rho_{\Og}(\x,\y),\]
   proving \eqref{3-2} on  the domain $\Og$.

%    \[ \wt \Og:=\Bl\{ r \bxi:\   \   0\leq r \leq (1+ 4n^{-2} C_1^2) h_\Og(\bxi),\  \ \bxi\in\sph\Br\},\]
% where $$h_\Og(\bxi): =\max\{ t\ge 0:\  \  t \bxi \in\Og\},\  \  \bxi\in\sph.$$
%  Clearly, $1\leq h_\Og(\bxi)\leq d$ and  $h_\Og(\bxi)\bxi\in\p \Og$ for each $\bxi\in\sph$. Moreover  $\wt \Og$ is a convex set satisfying
%\[ \Og\subset \wt\Og \subset B\Bl(0,  (1+ 4n^{-2} C_1^2)d\Br),\]
%and
%\[\da_{\wt\Og} (\x)=\f {4 C_1^2} {n^2} \|\x\|+\da_\Og(\x) \Bl(1+\f {4C_1^2} {n^2}\Br)\ge \f {C_1^2}{n^2},\]
%On the other hand, by Remez's  inequality,
%\begin{align*}
%\|Q\|_{\wt \Og} &=\max_{\bxi\in\sph} \max _{0\leq r \leq (1+4n^{-2} C_1^2)h_{\Og}(\bxi)} |Q(r\bxi) | \leq \max_{\bxi\in\sph} \max _{0\leq r \leq h_{\Og}(\bxi)  +4C_1^2 d n^{-2}} |Q(r\bxi) | \\
%&\leq C_d \max_{\bxi\in\sph} \max _{0\leq r \leq h_{\Og}(\bxi) } |Q(r\bxi) |=C_d.
%\end{align*}
%Thus, once  \eqref{3-2} is proved for the    convex body $\wt\Og$  and  the case of   $\da_{\wt \Og}(\x) \ge \f {C_1^2}{n^2}$, then using \eqref{2-6}, one has
%\[ |Q(\x)-Q(\y)|\leq C n \rho_{\wt\Og}(\x,\y) \leq C n  \rho_{\Og}(\x,\y),\]
%proving \eqref{3-2} for the domain $\Og$.
%
%
%In summary, we have shown that In the next subsection, we will  prove  that this problem can be  reduced to  a two dimensional problem.

\subsection{Reduction to a two-dimensional problem} \label{subsec4-2}  Set  $\da:= \da_{\Og} (\x)$. Assume  that
 $\f {C_1^2}{n^2}\leq \da\leq  \f 1 {100}$ and  $n\rho(\x,\y) \leq c_1$.
Our goal   in this subsection is to  reduce  \eqref{3-2}  to   a two-dimensional  inequality on a  planar  convex domain, under the assumption~\eqref{eqn:2-1relaxed}.

First, consider  the composition $\mathcal{T}_1:=\mathcal{Q}\circ \mathcal{S}$ of the shift operator $\mathcal{S}$ on $\R^d$ that moves $\x^\ast:=\x+\da \f {\x}{\|\x\|}$ to the origin, and   a rotation $\mathcal{Q}$ on $\R^d$ such that  $$\T_1(\0)=-\mathcal{Q} (\x^\ast) =\x_0:=( \0, \ell)\   \ \text{and}\  \ \T_1 (\y)=\mathcal{Q}(\y-\x^\ast)=(y_1,0,\dots, 0,  y_d)$$
for   $ \ell=\|\x^\ast\|\in [1,2d]$, and  some  $y_1, y_d\in \R$.
 Clearly,
  $\Og_1:=\mathcal{T}_1(\Og)$ is a convex domain in $\R^d$  such  that $\mathcal{T}_1 (\x) =(\0,\da)\in \Og_1$,  $\T_1 (\y)=(y_1,0,\dots, 0,  y_d)\in\Og_1$,
  $\mathcal{T}_1 (\x^\ast) =\0\in\p \Og_1$   and
$B(\x_0, 1)\subset \Og_1\subset B(\x_0, 2d)$.
Since  the metric $\rho_{\Og}$ is  shift-invariant and rotation-invariant,
we also have
\[ \rho_{\Og}(\x,\y)=\rho_{\Og_1}\Bl(\mathcal{T}_1(\x), \mathcal{T}_1(\y)\Br).\]
Thus, it is enough to prove the inequality \eqref{3-2} for $\Og=\Og_1$,  $\x=(\0,\da)$ and  $\y= (y_1,0,\dots, 0,  y_d)$ and $\rho=\rho_{\Og_1}$.

For the reminder of the proof, we will use  a slight abuse of notations that   $\x=(\0,\da)$ and  $\y= (y_1,0,\dots, 0,  y_d)$ and $\rho=\rho_{\Og_1}$.
We will  use  the notations  $\wt \z, \wt \u, \wt \w,\dots$ to  denote points in $\R^{d-1}$.  Sometimes we find it more convenient to    write a  point $\z\in\R^d$ in the form $\z=(\z_x, z_y)$ with $\z_x\in\R^{d-1}$ and $z_y\in \R$.

Now  we define a function $\wt f: B(\0, 1)_{\R^{d-1}}\to \R$ by
\[ \wt f (\wt \z) :=\min \Bl\{ t\leq \ell: \  \ (\wt \z, t)\in \Og_1\Br\},\   \ \wt \z \in B(\0, 1)_{\R^{d-1}}.\]
Since  $B(\x_0, 1) \subset \Og_1\subset B(\x_0, 2d)$  and $\Og_1$ is convex,   it is easily seen that for each  $\wt \z \in B(\0, 1)_{\R^{d-1}}$,
$\ell-2d\leq  \wt f(\wt \z)\leq \ell$, $(\wt \z, \wt f (\wt \z))\in\p \Og_1$ and  $$
\Og_1^\circ  \cap \Bl(\{\wt \z\}\times (-\infty, \ell] \Br)= \{\wt \z\} \times (\wt f(\wt \z), \ell].$$
This implies  that  $\wt f(\0)=0$, and
\begin{align}
\Og_1\cap \Bl(B(\0, 1)_{\R^{d-1}} \times [\ell-2d, \ell]\Br)&=\Bl\{ (\z_x, z_y):\    \ \|\z_x\|\leq 1\  \ \text{and}\  \    \wt f( \z_x) \leq z_y \leq \ell\ \Br\},\label{3-4a}\\
\p\Og_1\cap \Bl(B(\0, 1)_{\R^{d-1}} \times [\ell-2d, \ell]\Br)&=\Bl\{ (\z_x, z_y):\  \  z_y= \wt f( \z_x)\   \ \text{and}\  \  \|\z_x\|\leq 1\Br\}.\label{3-5a}
\end{align}
Since the set $ \Og_1\cap \Bl(B(\0, 1)_{\R^{d-1}} \times [\ell-d, \ell]\Br)$ is convex,  the function   $\wt f: B(\0, 1)_{\R^{d-1}} \to \R$ is   continuous and  convex. In particular, this implies that for each fixed  $\wt \z\in B(\0, 1)_{\R^{d-1}}\setminus \{0\}$,
 the function $t\in [-1,1]\to \wt f \Bl( t \f {\wt \z}{\|\wt \z\|}\Br)$ is convex, and hence
\begin{align*}
-2d\leq \f{\wt f \Bl(-\f {\wt \z}{\|\wt \z\|} \Br)-\wt f (\0)}{-1-0}\leq \f {\wt f (\wt \z)-\wt f (\0)}{\|\wt \z\|-0} \leq \f{\wt f\Bl( \f {\wt \z} {\|\wt \z\|}\Br)-\wt f(\0)}{1-0}\leq 2d.
\end{align*}
It follows that
\begin{equation}\label{6-3-a}
|\wt f (\wt \z) |\leq 2d  \|\wt \z\|,\   \ \forall \wt \z \in B(\0, 1)_{\R^{d-1}},
\end{equation}
and
\begin{align}
D_{\wt\bxi} \wt f(\0)\leq 2d\   \ \text{whenever $\wt\bxi\in\SS^{d-2}$}. \label{6-3}
\end{align}
A similar argument  also shows that  for any  $\|\wt \z \|\leq \f 1 {20d}$ and $\wt\bxi\in\SS^{d-2}$,
\begin{align}
D_{\wt\bxi} \wt f(\wt \z) &\leq \f {\wt f (\wt \z+(1-\f 1{20d} )\wt\bxi)-\wt f(\wt \z)}    { 1-\f 1{20d}} \leq \f {2d} {1-\f 1 {20d}} <3d.\label{6-2}
\end{align}
Combining \eqref{6-2} and  \eqref{6-3} with  \eqref{6-1}, we obtain
\begin{align} \sup_{\bg\in \p \wt f(\0) } \|\bg\|\leq 2d \  \  \text{and}\  \
\sup_{\bg\in \p \wt f(\wt \z) } \|\bg\|\leq 3d\   \    \    \text{whenever }\  \ \| \wt \z\|\leq \f 1 {20d}.\label{6-5}
\end{align}

Now we  set
\begin{equation} 
G_1:=\Bl\{ (\z_x, z_y)\in\R^d:\    \ \|\z_x\|\leq \f 1 {20d}\  \ \text{and}\  \    \wt f(\z_x) \leq z_y \leq \f 13 +\wt f'(\0) \cdot \z_x\ \Br\}. \label{7-5} 
\end{equation}
Here and throughout the proof,   $\wt f'(\wt\u)$ denotes a   subgradient of $\wt f$ at a point $\wt\u$
satisfying
$D_{\bde_1} \wt f(\wt \u)=\bde_1\cdot \wt f'(\wt\u)$, where $\bde_1=(1,0,\dots, 0)\in\R^{d-1}$.
By \eqref{3-4a},  \eqref{3-5a} and \eqref{6-5}, as $(2d+3d)\frac{1}{20d}<\frac13$, $G_1$ is a convex subset of $\Og_1$, and
\begin{equation}
 G_1\cap\p \Og_1= \Bl\{ (\z_x, \wt f (\z_x)):\    \ \|\z_x\|\leq \f 1 {20d}\Br\}.\label{7-6}
\end{equation}
Clearly, $\x=(\0,\da)\in G_1$.
We also have   $\y=(y_1,0,\dots,0, y_d)\in G_1$. Indeed, since $\rho(\x,\y)\leq \f {c_1}n$ and $c_1$ is a sufficiently small constant depending only on $d$, we have
$ |y_1|\leq \|\y-\x\|\leq 2\sqrt{2d} \rho(\x,\y) \leq \f 1 {20 d},$
and
$|y_d|\leq \da +|y_d-\da|\leq \da + 2\sqrt{2d} \rho(\x,\y)< \f 13 +y_1\wt f'(\0)\cdot \bde_1$, which, by   \eqref{7-5},  implies  that $\y\in G_1$.

Second,  we define $f: B(\0, 2)_{\R^{d-1}}\to \R$ by
\[ f( \wt{\bu}) := \wt f \Bl( \f { \wt{\bu}}{20 d}\Br) -\f 1 {20 d} \wt f'(\0) \cdot  \wt{\bu},\   \    \   \wt{\bu}\in B(\0, 2)_{\R^{d-1}}.\]
Clearly, $f$ is a convex function on $B(\0, 2)_{\R^{d-1}}$ satisfying $f(\0)=0$ and  $\0\in\p f(\0)$. By \eqref{6-3-a},  \eqref{6-5} and convexity, we have that
\begin{equation}
	\label{eqn:15bound}
	0\leq f(\wt{\bu}) \leq 3d \cdot \frac1{20d}< \f15
	\quad\text{for all}\quad \wt{\bu}\in B(\0, 1)_{\R^{d-1}}.
\end{equation}
Since
$ \p f(\wt \u)=\f 1 {20d} \p \wt f \Bl(\f {\wt \u} {20d}\Br)-\f  1{20d} \wt f'(\0)$,
it follows from \eqref{6-5} that
\begin{align}
\sup_{\bg\in \p f(\wt \u)} \|\bg \|\leq \f 15 +\f 1{20}<\f13.\label{7-7}
\end{align}

 Now  consider the non-singular  linear mapping $\mathcal{T}_2:\R^d\to\R^d$,
\[ \left[\begin{matrix}
\u_x \\
u_y
\end{matrix} \right]:= \mathcal{T}_2 \left[\begin{matrix}
\z_x \\
z_y
\end{matrix} \right] :=\left[ \begin{matrix}
20 d  \z_x \\
z_y-\wt f'(\0) \cdot \z_x
\end{matrix} \right] =\left[ \begin{matrix}
20 d \boldsymbol{I}_{d-1} & 0\\
-\wt f'(\0)& 1
\end{matrix}\right] \left[\begin{matrix}
\z_x \\
z_y
\end{matrix} \right],\]
where $\z_x,\u_x\in\R^{d-1}$,  $z_y,u_y\in\R$ and $\boldsymbol{I}_{d-1}$ denotes the $(d-1)\times (d-1)$ identity matrix.  Let $\Og_2:=\mathcal{T}_2(\Og_1)$. Clearly,  $\Og_2$  is a convex domain in $\R^d$,
$f\bigl((\T_2  \z)_x\bigr)=\wt f ( \z_x)-\wt f'(\0)\cdot \z_x$, and $\T_2$ maps the hyperplane $\z_y=\wt f'(\0) \cdot \z_x$ in the $\z$-space, which is the supporting plane to the convex set $\Og_1$ at the origin, to the coordinate plane $\u_y=0$ in the $\u$-space. Moreover, using $\|\wt f'(\0)\|\le 2d$, a straightforward calculation (for the lower bound, one can consider two cases depending on how big is $|\xi_d|$) shows  that  $$ \f1{16}<  \|\mathcal{T}_2^\ast\bxi\|< 25 d,\   \  \forall \bxi\in\sph,$$
which, using \eqref{eqn:metric after affine transform}, also implies that
\begin{equation}\label{3-13-a}
\f14 \rho_{\Og_1} (\z,\z')\leq \rho_{\Og_2} (\mathcal{T}_2 \z, \mathcal{T}_2\z')\leq 5\sqrt{d} \rho_{\Og_1} (\z,\z')\   \   \ \text{for any $\z, \z'\in\Og_1$}.
\end{equation}
Define $G_2:=\mathcal{T}_2(G_1)$. Then it is easily seen that
\[ G_2=\Og_2 \cap \Bl( B(\0, 1)_{\R^{d-1}}\times \Bl[0, \f13\Br]\Br)=\Bl\{ (\bu_x, u_y)\in\R^d:\  \  \bu_x\in B(\0, 1)_{\R^{d-1}},\   \  f(\bu_x) \leq u_y \leq \f 13\Br\},\]
and
\begin{align}\label{eqn:g2rep-boundary}
\Bl\{ (\wt \u, f(\wt \u)):\  \ \|\wt \u\|\leq 1\Br\} =\mathcal{T}_2(G_1\cap \p \Og_1)=\mathcal{T}_2(G_1)\cap \mathcal{T}_2(\p \Og_1) =G_2\cap \p \Og_2.
\end{align}
Clearly, $\mathcal{T}_2\x=\x=(\0,\da)\in G_2$,  and $\mathcal{T}_2\y=(a,0, \dots, 0, b)\in G_2$  lies in the  $2$-dimensional plane spanned by $\bde_1$ and $\bde_d$.

Now we define  a new   metric $\wt\rho_{\Og_2}$   on the domain $\Og_2$ as follows.
  Given $\wt\u_0\in B(\0, 1)_{\R^{d-1}}$ and  $\bg\in\p f(\wt\u_0)$, let
\[ H_{\wt \u_0}^{\bg} (\wt \u):= f(\wt \u_0)+\bg\cdot (\wt \u-\wt \u_0),\   \   \wt\u \in\R^{d-1},\]
so that $u_y=H^{\bg} _{\wt \u_0}(\u_x)$ is the supporting hyperplane of $\Og_2$ at the point $(\wt\u_0, f(\wt \u_0))$ in the $\u$-space.
 For
$\u=(\u_x, u_y)\in\Og_2$ and $\w=(\w_x, w_y)\in\Og_2$, we define
\begin{align}
\wt \rho_{\Og_2}(\u, \w):= \|\u-\w\|+ \sup_{H_{\wt \u_0}^{\bg}} \Bl| \sqrt{u_y-H_{\wt \u_0}^{\bg} (\u_x)}-\sqrt{w_y-H_{\wt \u_0}^{\bg }(\w_x)}\Br|.\label{7-8}
\end{align}
where the supremum is taken over all functions $H_{\wt \u_0}^{\bg}$ with
 $\wt\u_0\in B(\0, 1)_{\R^{d-1}}$  and $\bg \in \p f(\wt\u_0)$.
Clearly, $\wt\rho_{\Og_2}$ is a metric on $\Og_2$.
We  claim that there exists a constant $C_d>0$ depending only on $d$ such that \begin{equation}
 \wt \rho_{\Og_2}(\u, \w)\leq  C_d\, \rho_{\Og_2}(\u, \w)\   \   \ \text{for any  $\u,\w\in\Og_2$}.\label{7-9}
\end{equation}

To see this,  $\wt\u_0\in B(\0, 1)_{\R^{d-1}}$ and  $\bg\in\p f(\wt\u_0)$,  let
$ \bN_{\wt \u_0}^{\bg} := \f { ( -\bg, 1)} {\sqrt{ 1+ \|\bg\|^2}}.$
Clearly,  $\bN_{\wt \u_0}^{\bg}$ is the inward normal vector to the boundary $\p \Og_2$ of $\Og_2$ at the point $(\wt\u_0, f(\wt \u_0))$.
 By convexity, we have
$$\Og_2\subset \Bl\{ \w\in\R^d:\  \  \w \cdot \bN_{\wt\u_0}^{\bg}\ge (\wt \u_0, f(\wt \u_0))\cdot \bN_{\wt\u_0}^{\bg}\Br\}, $$
which in particular implies that  $$a_{\bN_{\wt \u_0}^{\bg} }:=\min_{\w\in\Og_2} \w\cdot \bN^{\bg}_{\wt\u_0}=(\wt \u_0, f(\wt\u_0))\cdot \bN^{\bg}_{\wt \u_0}. $$
Thus, for any $\w=(\w_x,w_y)\in\Og_2$,
\begin{align*}
&\w\cdot \bN^{\bg}_{\u_0} -a_{\bN^{\bg}_{\u_0}}:=\Bl(\w-(\u_0, f(\u_0))\Br)\cdot \bN^{\bg}_{\u_0} =\Bl((\w_x,w_y)-\bigl(\w_x, H^{\bg} _{\wt \u_0} (\w_x)\bigr)\Br)\cdot \bN^{\bg}_{\wt \u_0}\\
&=\f { w_y-H^{\bg}_{\wt \u_0}(\w_x)}{\sqrt{ 1+ \|\bg\|^2}}.
\end{align*}
It then follows by  \eqref{7-7}  that for any $\w=(\w_x,w_y)\in\Og_2$ and $\u=(\u_x,u_y)\in\Og_2$,
\begin{align*}
\rho_{\Og_2}(\w,\u)&\ge \Bl| \sqrt{ \w\cdot \bN^{\bg}_{\wt \u_0} -a_{\bN^{\bg}_{\wt \u_0}}}-\sqrt{ \u\cdot \bN^{\bg}_{\wt \u_0} -a_{\bN^{\bg}_{\wt \u_0}}}\Br|=\f{\Bl| \sqrt{\w_y-H_{\wt \u_0}^{\bg}(\w_x)}-\sqrt{\u_y-H^{\bg}_{\wt \u_0}(\u_x)}\Br|}{ \Bl( 1+\|\bg\|^2\Br)^{\f14} }\\
&\ge \Bl( \f 9{10}\Br)^{\f14} \Bl| \sqrt{\w_y-H^{\bg}_{\wt \u_0}(\w_x)}-\sqrt{\u_y-H^{\bg}_{\wt \u_0}(\u_x)}\Br|.
\end{align*}
This together with \eqref{2-5} and \eqref{3-13-a}  implies the  claim \eqref{7-9}.

Finally, let $G$ denote  the planar section  of $G_2$  that passes though the origin and the points $\T_2\x=(\0, \da)$ and $\mathcal{T}_2 \y=(a,0, \dots, 0, b)$; namely, $G$ is the intersection of $G_2$ with the coordinate plane spanned by $\bde_1$ and $\bde_d$.  With  a slight abuse of notations, we identify a point $(z_1, 0,\dots, 0, z_d)\in\R^d$  with the the point   $(z_1, z_2)\in\RR^2$, and write $f(x):=f(x,0,\dots,0)$ for $x\in [-2,2]$.
Then  $f:[-2,2]\to \R$ is a convex and continuous function satisfying that $f(0)=0$, $f_+'(0)=0$ (due to the choice of $\wt f'(\0)$),
$0\leq f(x)\leq \f 15$ (see~\eqref{eqn:15bound})
and $|f_{\pm}'(x)|\leq \f13$ for every $x\in [-1,1]$ (see~\eqref{7-7}). Moreover,  $G$ is a convex domain in $\R^2$ that can be represented as (see~\eqref{eqn:g2rep-boundary})
\begin{equation}
G=\Bl\{ (x,y):\  \ -1\leq x\leq 1,    \   f(x) \leq y\leq \f13\Br\}.
\end{equation}

Now the metric $\wt \rho_G$ on the domain $G$ is defined as follows.  For
$(x_1, y_1), (x_2, y_2)\in G$,
\begin{align}
&\wt \rho_G((x_1, y_1), (x_2, y_2)):=\|(x_1-x_2, y_1-y_2)\|+\sup_{L}
\Bl| \sqrt{y_1-L (x_1)}-\sqrt{y_2-L(x_2)}\Br|,\label{4-13}
\end{align}
where $\|\cdot\|$ denotes the Euclidean norm, and  the supremum is taken over all linear  functions $L:\R\to\R$ of the form
\begin{equation}\label{4-14} L(x):=f(x_0) +a (x-x_0),\  \ x_0\in [-1,1]\   \ \text{and}\   \ a=f_{+}'(x_0)\  \ \text{or}\  \  a=f_{-}'(x_0).\end{equation}
Clearly,  for each linear function $L$ given in \eqref{4-14},  $y=L(x)$ is a supporting line of the convex  set  $G$ at the point $(x_0, f(x_0))$. By convexity, this in particular implies that the  metric $\wt \rho_G$ is well defined.  On the other hand, however,  each supporting line $y=L(x)$ of $G$  is  the intersection of a supporting hyperplane of the convex domain $G_2$  with   the  coordinate plane spanned by $\bde_1$ and $\bde_d$ (this readily follows from the classical results on separation of convex sets such as~\cite{Ro}*{Th.~B, p.~83} or~\cite{Sc}*{Th.~1.3.7, p.~12}); more precisely,  for each linear function $L$ given in \eqref{4-14},    we can find   $\bg\in\p f(\wt\u_0)$ with $\wt\u_0=(x_0, 0,\dots, 0)$ such that $L(x)= H_{\wt \u_0}^{\bg} (x,0,\dots, 0)$,   implying that  for each $(x, y)\in G$,
$$ \u_y -H_{\wt \u_0}^{\bg} (\u_x)=y-L(x)\   \   \ \text{ with $\u:=(x,0,\dots, 0, y)\in G_2$}.$$
It follows that  for any  $(x_1, y_1), (x_2, y_2)\in G$,
$$\wt \rho_G((x_1, y_1), (x_2, y_2))\leq \wt \rho_{\Og_2}\Bl( ( x_1,0,\dots, 0, y_1), ( x_2,0,\dots, 0, y_2)\Br).$$
In particular,  this implies that
 \[ \wt\rho_G\Bl( (0,\da), (a,b)\Br)\leq \f {C_d c_1}n.\]

Putting the above together, taking into account the fact that the space $\Pi_n^d$ is invariant under non-degenerate affine transforms,   we reduce  the proof of  the inequality  \eqref{3-2} to a two-dimensional problem, which is formulated explicitly as follows.
\begin{proposition}\label{prop-6}
		Let $G\subset \R^2$ be a convex domain given by
	\[ G:=\Bl\{ (x,y)\in\R^2:\  \ -1\leq x\leq 1,\  \ f(x)\leq y\leq \f 13\Br\},\]
	where  $f:[-2,2]\to \R$ is a continuous and convex function satisfying   $f(0)=f_+'(0)=0$, 
	 $0\leq f(x)\leq \f 15$   and   $|f_{\pm}'(x)|\leq \f13$    for all $x\in [-1,1]$.
	Let  $ \f {C_1^2} {n^2}\leq \da \leq \f 1{100}$.  Then there exists an absolute constant $C>1$ such that for any   $Q\in\Pi_n^2$  and  $(a,b)\in G$,
	\[ |Q(0,\da)-Q(a,b)|\leq C n\cdot \wt \rho_G\big((0,\da),(a,b)\big)\|Q\|_G,\]
where  $\wt \rho_G$ denotes the metric  given in  \eqref{4-13}.
\end{proposition}

\subsection{Proof of the two-dimensional result}\label{subsec4-3}
This subsection is devoted to the proof of Proposition \ref{prop-6}.
 Assume that $ \f {C_1^2} {n^2}\leq \da \leq \f 1{100}$. Fix  $\x:=(0,\da)\in G$ and   $\y:=(a,b)\in G$  with $\wt \rho_G(\x,\y)\leq \f {c_1} n$. Then  $|a|\leq \|\x-\y\|\leq  \wt \rho_G(\x,\y) $, and
 \begin{align}\label{4-16}
 |\sqrt{b} -\sqrt{\da}|\leq \wt \rho_G(\x,\y) \leq \f {c_1}n \leq \f14 \sqrt{\da}.
 \end{align}
 In  particular, this implies  \begin{equation}\label{4-17}
 \   \ |b-\da|\leq \wt\rho_G(\x,\y) (\sqrt{b} +\sqrt{\da}) \leq 3\wt\rho_G(\x,\y) \sqrt{\da}.
 \end{equation}

  Let $Q\in\Pi_n^2$ be such that  $\|Q\|_G=1$.  Our goal is to show that \begin{equation}\label{4-15}
|Q(0,\da)-Q(a,b)|\leq C n\wt\rho_G(\x,\y).
\end{equation}

 For the proof of \eqref{4-15}, the non-trivial  case is when   $\max_{x\in [-1,1]} f(x) >\f \da 2$. In the case when     $\max_{x\in [-1,1]} f(x) \leq \f \da 2$,  \eqref{4-15}  can be deduced directly from  the one dimensional Bernstein inequality \eqref{1D-Bernstein}.  Indeed,   if  $\max_{x\in [-1,1]} f(x) \leq \f \da 2$, then
 \begin{align*}
I_0:=  \{ (x, \da):\  \ -1\leq x\leq 1\}\subset G\   \  \text{and}\  \ J_0:=\Bl\{ (a, y):\  \  \f \da 2  \leq y\leq \f13\Br\}\subset G.
 \end{align*}
  Applying   the Bernstein inequality \eqref{1D-Bernstein}  along  the horizontal line segment $I_0\subset G$ yields \begin{align*}
 |Q(0, \da)-Q(a,\da)|\leq |a| \max_{|x|\leq |a|} \Bl| \p_1 Q(x, \da)\Br| \leq \sqrt{2} n |a| \|Q\|_{I_0} \leq \sqrt{2} n \wt\rho_G (\x,\y),
 \end{align*}
 where the second step and the last step used the fact that $|a|\leq  \wt \rho_G(\x,\y)\leq \f12$.
 Similarly, applying  the Bernstein inequality \eqref{1D-Bernstein}  along  the vertical  line segment $J_0\subset G $, we obtain
 \begin{align*}
 &|Q(a, \da)-Q(a, b)|\leq   3\wt\rho_G(\x,\y) \sqrt{\da} \max_{y\in [ \f 9 {16} \da, \f {25}{16}\da]} \Bl| \p_2 Q(a, y)\Br|\\
 &\leq 48 n \wt\rho_G(\x,\y) \|Q\|_{J_0}\leq  48 n \wt\rho_G(\x,\y),
 \end{align*}
 where we used \eqref{4-16} and \eqref{4-17} in the first step.
 Thus, we have
\begin{align*}
|Q(\x)-Q(\y)|\leq |Q(0, \da)-Q(a,\da)|+|Q(a, \da)-Q(a, b)|\leq 50 n \wt\rho_G(\x,\y),
\end{align*}
proving  \eqref{4-15}  in the case when   $\max_{x\in [-1,1]} f(x) \leq \f \da 2$.

For the reminder of the proof, we always assume    $\max_{x\in [-1,1]} f(x) >\f \da 2$.
 We  need the following  geometric lemma, which is    Lemma~2.2 from~\cite{Pr} under  a slightly different assumption on the function $f$.

\begin{lemma}\label{lem-7} If $\max_{x\in [-1,1] } f(x) >\f \da 2$, then  we can find  a  positive constant $k$, a number $\xi\in[-1,1]\setminus\{0\}$, and a linear function $\ell(x):=\al x-\b$ with the following properties:
	\begin{align}
	&0<|\al|\leq \f13,\    0\leq \b\leq \f13,  \    \ell(\xi)= f(\xi),\    \ell'(\xi)= \ f_-'(\xi),\  \label{4-18}\\
	& \text{$f(x)\leq \f \da 2 +kx^2$ for all $x\in [-1,1]$},\label{4-19}\\
	&\f {\sqrt{\da+\b}}{|\al|}<\f 1 {\sqrt{k}}.\label{8-4}
	\end{align}
\end{lemma}

The proof of Lemma \ref{lem-7} is almost identical to that of  Lemma~2.2 from~\cite{Pr} (except  changes of some constants). For the sake of completeness, we include the proof below.
\begin{proof}
	
	Define
	\begin{equation}\label{eqn:k def}
	k:=\inf\left\{\tilde k\ge 0: \frac{\delta}{2}+\tilde kx^2\ge f(x)\   \text{for all}\   x\in[-1,1]\right\}.		
	\end{equation}
	Since $f(0)=0$ and $\max_{x\in [-1,1]} [f(x)-\wt k x^2]$ is a continuous function of $\wt k$,
	the infimum in \eqref{eqn:k def}  is well defined and  attained.  Since $\max_{x\in [-1,1] } f(x) >\f \da 2>f(0)=0$, it follows that $k>0$ and
	there  exists a number $\xi\in[-1,1]\setminus \{0\}$ such that
	 \begin{equation}\label{8-6}
	 \max_{x\in [-1,1]}(f(x)- kx^2)=f(\xi)- k\xi^2 =\f \da 2.
	\end{equation}
Next, define $\ell(x):=f(\xi) +\al  (x-\xi)=\al x-\beta$, where $\alpha=f'_-(\xi)$ and   \begin{equation}\label{4-23-0}
	\text{    $\beta=\al \xi -f(\xi)=\al \xi -k\xi^2-\frac\delta2$}.
	\end{equation}
Clearly,
	\begin{equation*}
\b< \al \xi  \leq  |\al|=|f_{-}'(\xi)|\leq \f13.
	\end{equation*}
	Since $f$ is convex and $f(0)=0$, we have $0\leq \b=f(0)-\ell(0)$,  which means \begin{equation}\label{8-5} 0< |\al|\leq \f13\  \ \text{and}\  \
	\frac\delta2\leq \al  \xi -k\xi^2.
	\end{equation}
%	Assume to the contrary  that $\b=0$. Then $f(0)-\ell(0)=\b=0$,  $f(\xi)-\ell(\xi) =0$,  and hence, using  the convexity of  $f-\ell$,  we have that   $f(x)-\ell(x)\leq   0$  for all $x\in [0, \xi]$, which  in turn implies that $f(x)=\ell(x)$ for all $x\in [0, \xi]$. However, this is impossible as $f_+'(0)=0$ and $\ell'(0) =\al>0$.

	Finally,  define   $p(x):=\frac\delta2+kx^2$ for  $x\in\R$.   Using \eqref{8-6}, we have  that   $f(\xi)=\ell(\xi)=p(\xi)$, whereas  by \eqref{eqn:k def}, and convexity, we have  $\ell(x)\le f(x)\le p(x)$ for all  $-1\le x\le 1$.   If  $\xi>0$,  then  \begin{equation*}
	\al=\ell'(\xi)=\lim_{x\to \xi-} \f {\ell(x)-\ell(\xi)}{x-\xi}
	\ge  \lim_{x\to \xi-} \f {p(x)-p(\xi)}{x-\xi}  =p'(\xi)=2k\xi,
	\end{equation*}
and similarly, if  $\xi<0$, then
	\begin{equation*}
\al=\ell'(\xi)=\lim_{x\to \xi+} \f {\ell(x)-\ell(\xi)}{x-\xi}
\leq  \lim_{x\to \xi+} \f {p(x)-p(\xi)}{x-\xi}  =p'(\xi)=2k\xi.
\end{equation*}
(If $\xi\in (-1,1)$, we in fact have equality $2k\xi=\al$ here.)
In either case, we have
 \begin{equation}\label{8-9}
0<\f {2\xi}{\al} \leq \f 1{k}.
\end{equation}
We can then   establish~\eqref{8-4}  as follows:
	\begin{align*}
	\frac{\sqrt{\delta+\beta}}{|\alpha|}& = \frac{\sqrt{\al \xi -k\xi^2+\frac\delta2}}{|\al|}
	<  \frac{\sqrt{2(\al \xi -k\xi^2)}}{|\al|} < \frac{\sqrt{2\al \xi}}{|\al|} \\
	&= \sqrt{\frac{2\xi}{\al}} \le \frac{1}{\sqrt{k}},
	\end{align*}
	where we used \eqref{4-23-0} in the first step,  the inequality \eqref{8-5} in the second step,  and the inequality \eqref{8-9} in the last step.
		This completes the proof of the lemma.
\end{proof}

Now we return to the proof of  \eqref{4-15}, assuming   that  $\max_{x\in [-1,1]} f(x) > \f \da 2$.   Let $k>0$ and $\xi\in [-1,1]\setminus \{0\}$ be given as in  Lemma \ref{lem-7}; that is,   \eqref{4-18},  \eqref{4-19} and  \eqref{8-4} are satisfied. To achieve the required bound, we will apply the one-dimensional Bernstein inequality~\eqref{1D-Bernstein} along a vertical line segment and along an appropriate part of the parabola $y=\wt p(x)$. We consider the following two cases depending on the leading coefficient $k$ of the parabola. \\

{\bf Case 1.}\  \  $0<k< 1$.

\vspace{5mm}

%
%Recall that $\y=(a, b)$ and $\|\x-\y\|\leq C \wt \rho(\x,\y)\leq \f c n$, where $c>0$ is a sufficiently small constant depending only on $d$. Note that
%$|a|\leq \|\x-\y\|\leq \f c n$, and
%\[ |\sqrt{\da}-\sqrt{b}|\leq \wt \rho(\x,\y) \leq \f c n <\f 14 \sqrt{\da},\]
%which in turn implies that
%\begin{equation}
%|\da-b|=(\sqrt{\da}+\sqrt{b}) |\sqrt{\da}-\sqrt{b}|\leq \f c n \sqrt{\da}.
%\end{equation}

This case is relatively easier to deal with.
Let $p(x):=\f \da 2 +k x^2$ and $\wt p(x)=\da +k x^2$.  By Lemma ~\ref{lem-7},  $f(x)\leq p(x)\leq \wt p(x)$ for all $x\in [-1,1]$.   Moreover,  if $|x|\leq \f 12$,
then
$\wt p(x)\leq \f 14 +\da <\f 13$. This means  that the parabola $$\Gamma_1:=\Bl\{ (x, \wt p(x)):\  \  |x|\leq \f 12 \Br\}$$
lies entirely in the domain $G$.
Now applying  the one dimensional  Bernstein inequality \eqref{1D-Bernstein} to  the univariate
polynomial $q(x):= Q(x, \wt p(x))\in\Pi_{2n}^1$, $x\in [-\f12, \f12]$,
and recalling $|a|\leq  \wt \rho_G(\x,\y) \le\f1{40} $,
we obtain
\begin{align}
|Q(0,\da) -Q(a, \wt p(a))| &=|q(0)-q(a)|\leq |a|\max_{|x|\leq a}|q'(x)| \leq 10 n |a| \|q\|_{[-\f12, \f 12]}\notag\\
& \leq 10n \wt \rho_G(\x,\y) \|Q\|_{\Gamma_1} \leq 10n \wt \rho_G(\x,\y).\label{4-27}
\end{align}

Next, we estimate the difference $|Q(a, b)-Q(a, \wt p(a))|$.
 We will apply   the one dimensional  Bernstein inequality \eqref{1D-Bernstein} along  the vertical line segment
\[ \Bl\{ (a, y):\  \ f(a)\leq y <\f 13\Br\}\subset G.\]
Since  $|a|\leq  \wt \rho_G(\x,\y) \leq \f {c_1} n$, we have  \[|\wt p(a)-\da|=k a^2 \leq (\wt \rho_G(\x,\y) )^2\leq \f {c_1^2}{n^2} \leq \f {\sqrt{\da}} {n}.\]
Furthermore, by
\eqref{4-17}, we have
\[ |b-\da|\leq 3\wt\rho_G(\x,\y) \sqrt{\da}\leq \f{3c_1} n \sqrt{\da}\leq \f {\sqrt{\da}} {n}.\]
This means that both $\wt p(a)$ and $b$ lie  in the interval $J_1:= [\da -\f {\sqrt{\da}}  n , \da+\f {\sqrt{\da}} n ]$, and moreover,
\[ |\wt p (a)-b|\leq |\wt p(a)-\da| +|b-\da| \leq (\wt \rho_G(\x,\y) )^2+3\wt\rho_G(\x,\y) \sqrt{\da}\leq 4\wt\rho_G(\x,\y) \sqrt{\da}.\]
It follows that
\begin{align*}
|Q(a, b)-Q(a, \wt p(a))|\leq |b-\wt p(a)| \max_{y\in  J_1} |\p_2 Q(a, y)|\leq 4\wt\rho_G(\x,\y) \sqrt{\da} \max_{y\in  J_1} |\p_2 Q(a, y)|.
\end{align*}
On the other hand, however,
\[ \da -\f {\sqrt{\da}} n -f(a) \ge \da -\f  {\sqrt{\da}} n -p(a) =\f \da 2 -\f { \sqrt{\da}}n -k a^2 \ge \f \da 3 -\Bl(\f {c_1} n\Br)^2 \ge \f \da 4,\]
where we used the fact that  $f(x)\leq p(x)$  for all $x\in [-1,1]$ in the first step,  and  the facts that $|a|\leq \f {c_1} n$ and $\da\ge \f {C_1^2}{n^2}$ in the last two steps.
This means that $f(a) \leq \f 34 \da$ and
\begin{equation}
 J_1 =\Bl[\da -\f {\sqrt{\da}}  n , \da+\f {\sqrt{\da}} n \Br] \subset \Bl[ f(a)+\f \da {4}, 2\da\Br] \subset \Bl[f(a), \f 13\Br].
\end{equation}
Thus, applying the one-dimensional Bernstein inequality \eqref{1D-Bernstein} to the polynomial $Q(a,\cdot)$ over the interval $[f(a), \f 13]$, we obtain
\begin{align*}
\sqrt{\da} \max_{y\in  J_1} |\p_2 Q(a, y)| \leq 6 n \max_{ y\in [f(a), \f 13]}|Q(a, y) |\leq 6n\|Q\|_G \leq 6n,
\end{align*}
implying
\begin{align}
|Q(a, b)-Q(a, \wt p(a))|\leq 4\wt\rho_G(\x,\y) \sqrt{\da} \max_{y\in  J_1} |\p_2 Q(a, y)| \leq 24 n\wt\rho_G(\x,\y).\label{4-29}
\end{align}
Now combining \eqref{4-27} with \eqref{4-29} yields the desired estimate \eqref{4-15}.
\vspace{3mm}

{\bf Case 2.}\  \  $k\ge 1$.

\vspace{3mm}

This is the more involved case.
For simplicity, we set  $\mu: =n\wt \rho(\x,\y)$, and
\[ B:=B_{\wt\rho}\Bl(\x, \f {\mu}n\Br):=\Bl\{ (x,y)\in G:\  \ \wt\rho_G\Bl( (0,\da), (x,y)\Br)\leq \f {\mu}n\Br\}.\]
Clearly,  $\mu\in (0, c_1]$ and  $\y\in B$.
Our aim is to show that \begin{equation}\label{4-30}
|Q(0,\da)-Q(a,b)|\leq C \mu
\end{equation}
for some absolute constant $C>0$.

Consider the   parallelogram
\[ E:=\Bl\{ (x,y)\in\R^2:\  \ y\ge 0, |\sqrt{y}-\sqrt{\da}|\leq \f {\mu} n,\  \ |\sqrt{y-\ell(x)}-\sqrt{\da+\b}|\leq \f {\mu}n\Br\},\]
where $\ell(x)=\al x-\b$ is the linear function from Lemma \ref{lem-7}.
From the definition of  $\wt\rho_G$, we have  $\x,\y\in B\subset E$.
 We further claim that  \begin{equation}\label{8-10}
E\subset R := \Bigg[ -\f {5\mu} {n\sqrt{k}}, \  \f {5\mu} {n\sqrt{k}}\Bigg]\times \Bigg[\da-\f {2\mu\sqrt{\da} }n, \  \da+\f {3\mu\sqrt{\da}}n\Bigg]\subset G.
\end{equation}
First, we prove  the relation $E\subset R$. We need to estimate the width and the height of the  parallelogram $E$.
 If  $(x,y)\in E$, then
$\sqrt{\da}-\f {\mu}n \leq \sqrt{y}\leq \sqrt{\da}+\f {\mu}n$, and
\[ \da-2\sqrt{\da}\f {\mu}n +\f {\mu^2}{n^2} \leq y \leq \da +2 \sqrt{\da} \f {\mu} n +\f {\mu^2}{n^2},\]
which, using  the assumptions that $\da\ge \f {C_1^2}{n^2}$ and $\mu\in (0, 2c_1)$, implies
 \begin{equation}\label{8-11}
y\in \Bigg[\da- \f { 2\sqrt{\da}\mu}n, \da+\f {3\sqrt{\da}\mu}n\Bigg].
\end{equation} This gives the desired estimate of   the height of the parallelogram $E$.

To estimate the width of $E$, we note that for  $(x, y)\in E$,
\[ \sqrt{\da+\b} -\f {\mu}n \leq \sqrt {y-\al x+\b} \leq \sqrt{\da+\b} +\f {\mu}n,\]
implying that
\[  \da -\f {2\mu}n \sqrt{\da+\b} +\f {\mu^2} {n^2} \leq y-\al x\leq \da+\f {2\mu}n \sqrt{\da+\b}+\f {\mu^2} {n^2}.\]
Since
\[ \f {\mu^2}{n^2}\leq \f {\mu \sqrt{\da}}n \leq \f {\mu} n \sqrt{\da+\b},\]
it follows that
\[ y- \da-\f {3\mu}n \sqrt{\da+\b}  \leq \al x\leq y- \da +\f {2\mu}n \sqrt{\da+\b},\]
which, using \eqref{8-11} and Lemma ~\ref{lem-7},  implies
\[ |x|\leq   \f {5\mu\sqrt{\da+\b} } {n|\al|} \leq \f {5\mu} {n\sqrt{k}}.\]
This bounds the width of $E$,  and hence completes the proof of the relation $E\subset R$.

To complete the proof of the claim \eqref{8-10}, it remains to prove the relation  $R\subset G$. Recall that $\mu\in (0, c_1]$. If $c_1>0$ is small enough, then
$\da +\f {3\mu \sqrt{\da}}n<\f13$  and $\f {5\mu}{n\sqrt{k}}\leq \f 12$, which means that
\begin{equation}\label{4-33}
R\subset \Bl\{ (x, y):\  \ |x|\leq \f {5\mu} {n\sqrt{k}}\leq \f12,\   \da-\f {2\mu\sqrt{\da}}n\leq y\leq \f13\Br\}.
\end{equation}
As   before, set $p(x):=\f \da 2+kx^2$. By Lemma ~\ref{lem-7}, we have $f(x)\leq p(x)$ for all $x\in [-1, 1]$. If, in addition,   $|x|\leq \f {5\mu}{n\sqrt{k}}$ and $c_1>0$ is small enough,  then
\[ f(x)\leq  p(x)\leq \f \da 2+ \f {25\mu^2}{n^2}\leq \da-\f {2\mu\sqrt{\da} }n.\]
Thus,  $R$ is a rectangle in $G$ that  lies above the parabola $\{(x, y):\  \ y=p(x),\  \  |x|\leq \f {5\mu}{n\sqrt{k}}\}$.
Using  \eqref{4-33}, we then obtain $R\subset G$, which  completes the proof of the claim \eqref{8-10}.

% Recall that  $\y=(a,b)\in B\subset E\subset R$.
%Thus,  by  \eqref{8-10}, it is sufficient to prove that
%\begin{equation}\label{4-34}
%|Q(0,\da)-Q(x,y)|\leq C \mu,\    \    \  \text{for any $(x,y)\in R$}.
%\end{equation}

Now we turn to the proof of \eqref{4-30}.
With $\wt p(x):=\da+kx^2$, by Lemma ~\ref{lem-7},
if $|x|<\f 1 {2\sqrt{k}}$,  then \[
f(x)\leq p(x)<\wt p(x) \leq \da +\f14<\f13.\]
This means  that the parabola
\[\Gamma_2:=\Bl\{ (x,\wt p(x)):\  \ |x|\leq \f t {2\sqrt{k}}\Br\}\]
lies entirely in the domain $G$.
Now we  bound the term on the left hand side of \eqref{4-30}  as follows:
\begin{equation}\label{4-34}
|Q(0,\da)-Q(a,b)|\leq \Bl|Q(0,\da)-Q(a,\wt p(a))\Br|+\Bl|Q(a,\wt p(a))-Q(a,b)\Br|.
\end{equation}
These last two differences can be estimated by  applying the  one-dimensional Bernstein inequality \eqref{1D-Bernstein} along the parabola $y=\wt p(x)$ and along the vertical line $x=a$ respectively.

Indeed, to estimate $\Bl|Q(0,\da)-Q(a,\wt p(a))\Br|$, we define  the polynomial
$$q(t):=Q\Bl( \f {t}{2\sqrt{k}}, \wt p\Bl(\f {t}{2\sqrt{k}} \Br)\Br),\  \  t\in [-1,1],$$  which  is the restriction of $Q$ on $\Gamma_2$.
 Clearly,
$q\in\Pi_{2n}^1$ and  $\|q\|_{[-1,1]}=\|Q\|_{\Gamma_2}\leq \|Q\|_G=1$.
Recall that
\[ \y=(a,b)\in R = \Bigg[ -\f {5\mu} {n\sqrt{k}}, \  \f {5\mu} {n\sqrt{k}}\Bigg]\times \Bigg[\da-\f {2\mu\sqrt{\da} }n, \  \da+\f {3\mu\sqrt{\da}}n\Bigg]\subset G.\]
Thus, setting  $t_a:=2\sqrt{k}a$, we have
\[ |t_a|=2\sqrt{k}|a|\leq  \f {10 \mu} n <\f19.\]
It follows by the Bernstein inequality \eqref{1D-Bernstein}  that
\begin{align}\label{4-35}
\Bl| Q(0, \da)-Q(a, \wt p(a))\Br|&=\Bl| q(0)-q (t_a)\Br|\leq   \f {10\mu}{n}\|q'\|_{[-\f19, \f19]}\leq   40 \mu\|q\|_{[-1, 1]}
\leq 40 \mu.
\end{align}

It remains to estimate  the difference   $|Q(a,\wt p(a))-Q(a,b)|$.  Set
\[ J_n(\da):= \Big[\da-\f {2\mu\sqrt{\da} }n, \  \da+\f {3\mu\sqrt{\da}}n\Big].\] Clearly,  $b\in J_n(\da)$. Also, we have
\[ |\wt p(a)-\da| =ka^2 \leq \f {25 \mu^2} {n^2}\leq \f { \mu \sqrt{\da}} n,\]
implying   $\wt p(a) \in  J_n(\da)$. Thus,
\begin{align}
& \Bl|Q(a, \wt p(a)) -Q(a,b)\Br|\leq  \f {5\mu\sqrt{\da}} n  \max_{ u\in  J_n (\da)}  \Bl|\p_2 Q(a, u)\Br|.\label{4-36}
\end{align}
To estimate the term $\max_{ u\in  J_n (\da)}  \Bl|\p_2 Q(a, u)\Br|$, we will apply the Bernstein inequality \eqref{1D-Bernstein} to the restriction of $Q$ to the   vertical line segment
\[I(a):=\Bl\{(a,y):\  \ f(a)\leq y \leq \f 1 3\Br\}=\{a\}\times \Bl[f(a), \f13\Br].\]
By  \eqref{8-10}, we have  $$ \{a\}\times J_n(\da)\subset I(a)\subset G.$$    We  need to ensure that the left endpoint   of the interval $J_n(\da)$ is sufficiently ``far'' from the left endpoint $f(a)$  of the interval $[f(a), \f13]$. Indeed, recalling  that  $p(a)=\f\da 2 +ka^2\ge f(a)$, we have
\begin{align*}
\da-\f {2\mu} n \sqrt{\da}-f(a) \ge \da-\f {2\mu} n \sqrt{\da}-\Bl(\f \da 2+ka^2\Br) \ge \f \da 2 -\f {2\mu} n \sqrt{\da} -\f {25 \mu^2} {n^2} \ge \f \da {4},
\end{align*}
implying  that $f(a)<\f 34 \da$ and
\[ J_n(\da) \subset \wt {J}_n:=\Bl[f(a)+\f14 \da , 2\da\Br] \subset \Bl[f(a), \f13 \Br].\]
Thus, by \eqref{4-36} and the Bernstein inequality \eqref{1D-Bernstein}, we deduce
\begin{align}
& \Bl|Q(a, \wt p(a)) -Q(a,b)\Br|\leq  \f {4\mu\sqrt{\da}} n  \max_{ y\in \wt J_n }  \Bl|\p_2 Q(a, y)\Br|\notag\\
&\leq  40 \mu \max_{ y\in [f(a), 3^{-1}]  }  \Bl| Q(a, y)\Br|=40\mu\|Q\|_{I(a)}\leq 40 \mu.\label{4-37}
\end{align}
Now combining \eqref{4-34}, \eqref{4-35} with \eqref{4-37}, we obtain the desired estimate \eqref{4-30}.

\section{Doubling property}\label{sec5:lec}

Our main goal in this section is to prove \cref{thm:doubling}. The main idea is representing $\rho$-neighborhoods as the intersection of ``strips'' and tracking the size of each strip under a dilation. We need the following elementary lemma.
\begin{lemma}\label{lem-4-1} If  $s, t\ge 0$,  $h>0$ and $|\sqrt{s}-\sqrt{t}|\leq 2h$, then
	\begin{equation}\label{15}
	|\sqrt{s}-\sqrt{\wt{s}}|\leq h,\   \ \text{
		where $\wt s:= s+\frac 14 (t-s)$.}
	\end{equation}
\end{lemma}
\begin{proof}  If  $ \sqrt{s}\ge 2h$, then
	\[ s+4h^2-4h\sqrt{s}=(\sqrt{s}-2h)^2 \leq t \leq (\sqrt{s}+2h)^2=s+4h^2+4h\sqrt{s},\]
	and hence, we have
	$h^2-h\sqrt{s}\leq \frac 14 (t-s)\leq h^2+h\sqrt{s},$
	implying
	\begin{align*}
	\sqrt{s}-h\leq \sqrt{	s+h^2-h\sqrt{s}}\leq \sqrt{\wt s } \leq \sqrt{s+h^2+h\sqrt{s}} \leq \sqrt{s}+h.
	\end{align*}
If $0\leq \sqrt{s} <2h$, then
\[-s\leq t-s= (\sqrt{t}-\sqrt{s}) (\sqrt{t}+\sqrt{s})\leq 2h (2\sqrt{s}+2h)= 4h\sqrt{s} +4h^2,\]
which implies
	\[ \frac 34 s\leq  \wt s  \leq s+h\sqrt{s} +h^2\leq (\sqrt{s}+h)^2,\]
and hence
	\[ -h\leq  -(2-\sqrt{3})h\leq \Bl(\f {\sqrt{3}}2-1\Br) \sqrt{s}\leq \sqrt{\wt s} -\sqrt{s}\leq h.\]
	In either case, we prove   \eqref{15}.
\end{proof}

\begin{proof}[Proof of \cref{thm:doubling}]
	Given $\x\in\Og$, $\bxi\in\S^{d-1}$ and $h>0$,
	define a ``strip''
	\[ S(\Og, \x, \bxi, h):=\Bl\{\y\in\Og:\  \ \Bl|\sqrt{\x\cdot \bxi -a_\bxi}-\sqrt{\y\cdot \bxi-a_\bxi} \Br| \leq h\Br\}.\]
%	Geometrically, $S(\Og, \x, \bxi, h)$ is the intersection of $\Og$ and  a strip bounded by two half-spaces: $$a_{\bxi}+\Bl( \max\{ \sqrt{\x\cdot \bxi-a_{\bxi}}-h,0\}\Br)^2  \leq \y\cdot \bxi\leq a_{\bxi} +(\sqrt{\x\cdot\bxi-a_{\bxi}} +h
%)^2.$$
%In particular, this implies that $S(\Og, \x, \bxi, h)$ is a convex body in $\R^d$.
		Let $\phi(\z):=\x+4(\z-\x)$ denote the dilation of $\z$  with respect to the center $\x$ with ratio $4$. We claim that
	\begin{equation}\label{17}
	S(\Og,\x, \bxi,2h)\subset 	\phi\Bl( S(\Og, \x, \bxi, h)\Br),\   \ \x\in\Og,\   \ \bxi\in\S^{d-1},\  \ h>0.
	\end{equation}
	To see this, let  $\y\in S(\Og,\x, \bxi,2h)$,  and let $s:=\x\cdot\bxi-a_\bxi$ and $t:=\y\cdot \bxi-a_\bxi$. Then  $s, t\ge 0$,  $|\sqrt{s}-\sqrt{t}|\leq 2h$. Setting $\z:=\x+\f14(\y-\x)$, we have $\z\cdot \bxi-a_{\bxi}=s+\f14(t-s)$, and hence by  Lemma~\ref{lem-4-1},
	$$|\sqrt{\z\cdot \bxi-a_{\bxi}}-\sqrt{\x\cdot \bxi-a_{\bxi}}|\leq h.$$
	On the other hand, however,  by convexity, we have 	
	$\z\in [\y,\x]\subset \Og$. Thus,
	 $\z\in S(\Og, \x, \bxi, h)$.  The claim \eqref{17} then follows since $\y=\phi(\z)$.

Now using  \eqref{17},
% and noticing that
%	\begin{align*}\label{16}
%	B_\rho (\x,h)=\bigcap_{\bxi\in\S^{d-1}} S(\Og, \x,\bxi,h),
%	\end{align*}
	  we obtain
	\begin{align*}
	B_\rho (\x,2h)&=\bigcap_{\bxi\in\S^{d-1}} S(\Og, \x,\bxi,2h)\subset \bigcap_{\bxi\in\S^{d-1}}\phi\Bl( S(\Og, \x,\bxi,h)\Br)=\phi\Bl( \bigcap_{\bxi\in\S^{d-1}} S(\Og, \x,\bxi,h)\Br)\\
	&=\phi\Bl( B_\rho (\x, h)\Br),
	\end{align*}
	implying
	\begin{align*}
	\lambda_d(B_\rho(\x, 2h))\leq \lambda_d\left(\phi\left( B_\rho(\x,h)\right)\right)=\lambda_d(\x+4 (B_\rho(\x, h)-\x))=4^d \lambda_d(B_\rho(\x,h)).
	\end{align*}
\end{proof}

\section{Fast decreasing polynomials} \label{sec6:lec}

The main goal in this section is to prove  \cref{thm:polynomial construction}.
The required polynomial $P$ in  \cref{thm:polynomial construction}  will be  built as a product of polynomials from the following lemma.
\begin{lemma}\label{lem:resolve} Given  an integer $n\ge 2$ and any    $\x,\y\in\Omega$  with    $\rho(\x,\y)\ge \f {2\pi\sqrt{2d}} {n-1}$,   there exists a polynomial $P\in\P_{n}^d$ such that  $P(\y)=1$, $P(\x)=0$, and $0\le P(\z)\le 1$ for all $\z\in\Omega$.
\end{lemma}

\begin{proof} Fix   $\x,\y\in\Omega$  satisfying     $\rho(\x,\y)\ge \f {2\pi\sqrt{2d}} {n-1}$.
	Let $\bxi\in\sph$ be such that
	\begin{equation}\label{4-1}
	%|\sqrt{\x\cdot\bxi-a_\bxi}-\sqrt{\y\cdot \bxi-a_\bxi}|\leq  
	\rho(\x,\y)\leq 2|\sqrt{\x\cdot\bxi-a_\bxi}-\sqrt{\y\cdot \bxi-a_\bxi}|,
	\end{equation}
	where, as usual, $a_{\bxi}=\min_{\z\in\Og} \z\cdot \bxi$.
	We also define  $b_\bxi:=\max_{\z\in\Og}\z\cdot \bxi$.
Since $\Og\subset B(\0, d)$, we have
\begin{equation}\label{4-2-0}
	-d\leq a_\bxi<b_{\bxi}\leq d, \quad\text{so}\quad b_\bxi-a_\bxi\le2d.
\end{equation}
By symmetry, we may assume that $\x\cdot \bxi \leq \y\cdot \bxi$.
	
	Now we define
	\begin{equation}\label{4-3-0}p(\z):=\f 2 {b_\bxi-a_\bxi} \Bl( \z\cdot \bxi -\f {a_\bxi+b_\bxi}2\Br),\   \  \z\in\R^d.\end{equation}
	Since  $\x\cdot \bxi \leq \y\cdot \bxi$ and  $ \z\cdot\bxi\in [a_\bxi, b_{\bxi}]$ for every $\z\in \Og$,   we have that  $p(\Og)=[-1,1]$ and $-1\le p(\x)\leq p(\y)\le 1$.
	We further claim
	\begin{align}\label{5-3}
	|\arccos p(\x) -\arccos p(\y)| \ge \f 1 {\sqrt{2d}}  \rho(\x, \y),
	\end{align}
	which in particular implies that $p(\x)< p(\y)$.
	 Indeed, 
	 \begin{align*}
	 	|\arccos p(\x)-\arccos p(\y)| &=\int_{p(\x)}^{p(\y)} \f 1{\sqrt{1-t^2}}dt \\ &\ge \f 1 {\sqrt{2}} \int_{p(\x)}^{p(\y)} \f {1}{\sqrt{1+t}}dt \\
	 	&=2\f  {
	 		\sqrt{\y\cdot \bxi -a_\bxi} -\sqrt{ \x\cdot \bxi -a_\bxi}}{\sqrt{b_\bxi-a_{\bxi}}},
	 \end{align*}
 	which leads to the claimed~\eqref{5-3} directly by~\eqref{4-1} and~\eqref{4-2-0}.

	Finally, we construct  a polynomial $P\in\Pi_{n}^d$ with the stated properties. Let $0\leq \ta_2<\ta_1\leq \pi $ be such that  $p(\x)=\cos\ta_1$ and $p(\y)=\cos\ta_2$.
	Let $m\ge 3$  be the integer such that $\f {2\pi} m < \ta_1-\ta_2\leq \f {2\pi }{m-1}$, and let  $1\leq \ell \leq m$  be the   integer  satisfying
	$ \f {(\ell-1) \pi} m \leq \ta_2 < \f {\ell \pi}m$.
	Then  $  \cos \f {\ell\pi}m \leq p(\y)=\cos\ta_2\leq \cos \f {(\ell-1) \pi} m,$ and
	$ \ta_1 \ge \ta_2 +\f {2\pi} m  \ge \f {(\ell+1)\pi}m,$
	which, in particular,  implies that $1\leq \ell \leq m-1$ and
	\begin{equation}\label{5-6} -1\leq p(\x) \leq \cos \f {(\ell+1)\pi} m \leq \cos \f {\ell\pi}m\leq p(\y)\leq 1.\end{equation}
	We  consider the mapping $\varphi:[-1,1]\to \R$ given by
	\[ \varphi(t):= \f {p(\y)-t}{p(\y)-p(\x)} \cos \f {(\ell+1)\pi}m +\f { t-p(\x)}{p(\y)-p(\x)} \cos \f {\ell \pi} m,\  \ t\in [-1,1].\]
	It is easily seen that $\varphi$  maps the interval $ [p(\x), p(\y)]$  onto  the interval
	$\Bl[\cos \f {(\ell+1)\pi} m,   \cos \f {\ell\pi}m\Br]$, and      $\varphi([-1,1])\subset [-1,1].$
	Since $p(\Og)\subset [-1,1]$, 	we may define,  for each  $\z\in\Og$,
	$$u(\z):=\varphi(p(\z))=\f {p(\y)-p(\z)}{p(\y)-p(\x)} \cos \f {(\ell+1)\pi}m +\f { p(\z)-p(\x)}{p(\y)-p(\x)} \cos \f {\ell \pi} m.$$
	Clearly, $u\in \Pi_{1}^d$, 	
	\begin{equation}\label{5-7}
	u(\x)=\cos \f {(\ell+1)\pi}m<u(\y)=\cos \f {\ell \pi} m,
	\end{equation}
	and   $u(\Og)\subset [-1,1]$.

	Now  we define
	$$P(\z):= \f{1+ (-1)^{\ell} T_m (u(\z))}2,\   \ \z\in \Og,$$
	where $T_m(t):=\cos (m\arccos t)$,  $t\in [-1,1]$ denotes the $m$-th degree  Chebyshev polynomial of the first kind.
	Since $u\in\Pi_{1}^d$ and $u(\Og)\subset [-1,1]$, $P$ is a well-defined algebraic polynomial of degree at most $m$ on the domain $\Og$. Clearly,  $\|P\|_{\Og} \leq 1$, and by \eqref{5-7},
	$$ P(\x)=\f {1+(-1)^\ell T_m \Bl(\cos \f {(\ell+1)\pi}m\Br)} 2=0\   \   \ \text{
		and}\  \    P(\y)=\f {1+ (-1)^\ell T_m \Bl( \cos \f {\ell \pi}m\Br)}2 =1.$$

	It remains to verify that $m\leq n$ provided that $\rho(\x,\y)\ge \f {2\pi\sqrt{2d}} {n-1} $. Indeed,
	$$ \ta_1-\ta_2=|\arccos p(\x) -\arccos p(\y)|\leq \f {2\pi} {m-1},$$
	which,  using \eqref{5-3}, implies
	$$m\leq  \f {2\pi\sqrt{2d}}  {\rho(\x, \y)}+1\leq  n.$$

\end{proof}

%\begin{theorem}\label{thm:polynomial construction}
%	For any $y\in \Omega$ there exists $P\in\P_{\alpha n,d}$ such that $P(y)=1$ and
%	\begin{equation}\label{eqn:fast decreasing}
%	0\le P(x)\le C \exp(-c\sqrt{\alpha n\rho(x,y)}), \qtq{for any} x\in\Omega.
%	\end{equation}
%\end{theorem}

\begin{proof}[Proof of \cref{thm:polynomial construction}]
	Let $\x\in \Omega$ and $n\in\mathbb{N}$. Let $\rho=\rho_{\Og}$ be the metric on $\Og$  defined in  \eqref{eqn:def-dubiner}. Note that  $\max_{\z,\z'\in\Og} \rho(\z,\z') \leq \max_{\z,\z'\in\Og}\sqrt{\|\z-\z'\|}\leq \sqrt{2d}$.
	Our goal is to construct a polynomial  $P\in\P_{n}^d$ such that   $P(\x)=1$ and
	\begin{equation}\label{eqn:fast decreasing-1}
	0\le P(\z)\le C \exp(-c\sqrt{ n\rho(\x,\z)}), \qtq{for any} \z\in\Omega.
	\end{equation}
	where  $C>1$ and $c\in (0, 1)$ are constants depending only on $d$.
	Such a  polynomial $P$   will be  built as a product of the polynomials from  \cref{lem:resolve}.

			By \cref{lem:resolve},   for each $\bog \in\Og\setminus\{\x\}$,  there exists an algebraic  polynomial $p_{\bog}$ in $d$-variables of total degree
	$$\deg (p_{\bog} )\leq \f {2\pi\sqrt{2d} } {\rho(\x, \bog)}+1\leq \f {11\sqrt{d} } {\rho(\x, \bog)}$$
	such that    $p_{\bog}(\x)=1$,  $p_{\bog}(\bog)=0$ and $0\leq p_{\bog}(\z)\leq 1$ for all $\z\in \Og$.   Let  $C_\ast=C_\ast(d)>1$ denote  the   constant given in \cref{thm:norm control}. Set   	
	$\al:=(22 \sqrt{d} C_\ast)^{-1}\in (0, \f12)$.   Then
	$  C_\ast \deg(p_{\bog})\cdot \big(\al  \rho(\x, \bog)\big)\leq \f12 $, and hence, using \cref{thm:norm control}, we conclude that
	\begin{align}\label{5-8}
	|p_{\bog}(\z) |=|p_{\bog}(\z)-p_{\bog}(\bog)|\leq \f 12 \   \ \text{whenever $\z\in\Og$ and  $\rho(\z, \bog) \leq \al  \rho(\x, \bog)$.}
	\end{align}

	Next, let $L>1$ be  a  large constant   depending  only on $d$, which  will be  specified later. Without loss of generality, we may assume that $n>L$, since otherwise the stated result with $P\equiv 1$ holds trivially.
	Set  $n_1:=n/L$, and consider a partition $\Og:=\bigcup_{j=0}^m\Og_j$ of $\Og$, given by
	\begin{align*}
	\Og_0:&=\{ \z\in\Og:\  \ \rho(\z, \x) \leq n_1^{-1}\}=B_\rho(\x, n_1^{-1}),\\
	\Og_j:&=\Bl\{\z\in\Og:\  \ 4^{j-1} n_1^{-1} < \rho(\z, \x) \leq 4^j n_1^{-1} \Br\}=B_\rho(\x, 4^j n_1^{-1})\setminus B_\rho(\x, 4^{j-1}n_1^{-1}),\\
	&j=1,2,\dots, m,
	\end{align*}
	where  $m\ge 1$ is the largest integer such that $4^{m-1}  \leq \sqrt{2d} n_1 $.
	%		 Here and throughout the proof, we assume without loss of generality  that the integer $n$  is sufficiently large so that $n_1 \sup_{\z, \z'\in\Og} \rho(\z, \z')>1$  since otherwise  the stated properties  hold trivially for the polynomial
	%		\[ P(\z):=1+\f {(\z-\x)\cdot \bxi}{2b_{\bxi}},\   \ \z\in\Og, \    \  \text{with}\  \ \bxi\in\sph\  \ \text{and}\  \ b_{\bxi} =\max_{\z\in\Og} \z\cdot \bxi.\]
	For each integer $1\leq j\leq m$, let $\Ld_j$ be a maximal $\al 4^{j-1} n_1^{-1}$-separated subset of $\Og_j$ w.r.t. the metric $\rho$.  Then
	\begin{equation}\label{5-9}
	\Og_j\subset \bigcup_{\bog\in\Ld_j} B_\rho(\bog, \al 4^{j-1} n_1^{-1}),\   \ j=1,2,\dots, m.
	\end{equation}
	Moreover, since $\Og_j\subset B_\rho(\x, 4^j n_1^{-1})$,    by  the doubling property stated in  \cref{thm:doubling} and the standard volume comparison argument,  it follows that  $$\# \Ld_j\leq C_d:= \Bl( 2+\f 8 \al\Br)^{2d},\   \ 1\leq j\leq m.$$
	
	Now we define
	\begin{equation}\label{5-10}
	P(\z):=\prod_{j=1}^m \Bl( \prod_{\bog\in\Ld_j} p_{\bog}(\z)\Br)^{2^j},\   \ \z\in \Og.
	\end{equation}
	Clearly, $P$ is an algebraic polynomial in $d$ variables such that $P(\x)=1$, $0\leq P(\z)\leq 1$ for all $\z\in\Og$ and
	\begin{align*}
	\deg (P) &=\sum_{j=1}^m 2^j \sum_{\bog\in\Ld_j} \deg(p_{\bog})  \leq \sum_{j=1}^m  2^j \sum_{\bog\in\Ld_j} \f {11\sqrt{d}} {\rho(\x, \bog)}
	\leq 11\sqrt{d}\sum_{j=1}^m 2^j (4^{j-1} n_1^{-1})^{-1}\#\Ld_j \\
	&\leq \f {44\sqrt{d} C_d n}{L} \sum_{j=1}^\infty 2^{-j}=\f {44\sqrt{d} C_d n}{L}.
	\end{align*}
	Now specifying   $L= 44\sqrt{d} C_d$, we have $P\in\Pi_{n}^d$.

	To complete the proof, it suffices  to show that $P$ satisfies the condition  \eqref{eqn:fast decreasing-1}.
	 Since $\Og=\bigcup_{j=0}^m \Og_j$, we only need to  verify the estimate \eqref{eqn:fast decreasing-1} for each   $\z\in\Og_j$  and  $0\leq j\leq m$.
	Since $0\leq P(\z)\leq 1$ for every $\z\in\Og$, the estimate 	\eqref{eqn:fast decreasing-1} holds trivially with $C\ge \exp(c L^{-\f12}) $ and $c\in (0, 1)$ if $\z\in \Og_0=B_\rho(\x, n_1^{-1})$.
	
	Now assume that  $\z\in\Og_j$ for a fixed  $1\leq j\leq m$.
	By \eqref{5-9}, there exists $\bog_\z\in\Ld_j$ such that
	\[ \rho(\z, \bog_\z) \leq \al 4^{j-1} n_1^{-1} <\al \rho(\bog_\z, \x),\]
	which, using \eqref{5-8}, implies that $0\leq p_{\bog_\z}(\z)\leq \f12$.  Since $0\leq p_{\bog}(\z)\leq 1$ for every $\bog\in\Og\setminus\{\x\}$, we obtain from \eqref{5-10} that
	\begin{align*}
	0\leq P(\z) \leq \Bl( p_{\bog_\z} (\z) \Br)^{2^j} \leq 2^{-2^j}=e^{-2^j\log 2} \leq \exp \Bl( -c\sqrt{n \rho(\x,\z)}\Br),
	\end{align*}
	where  $c:=\f {\log 2} {\sqrt{ (1+\f \al 4) L}}$, and the last inequality holds because
	\begin{align*}
	n \rho(\x,\z)\leq n \Bl( \rho(\x, \bog_\z) +\rho(\z, \bog_\z)\Br) \leq n \Bl( 4^j n_1^{-1} + \al 4^{j-1} n_1^{-1}\Br) \leq c^{-2} (\log 2)^2 4^j.
	\end{align*}
	This  proves \eqref{eqn:fast decreasing-1} and hence completes the proof of \cref{thm:polynomial construction}.
\end{proof}

\section{Optimal meshes}\label{sec 7:lec}

This section is devoted to the proof of Theorem \ref{thm:optimal mesh}.
	Since optimal meshes  are invariant under nonsingular  affine transforms,  by  John's theorem on inscribed ellipsoid of the largest volume~\cite{Sc}*{Th.~10.12.2, p.~588}, without loss of generality, we may assume that $\Og$ is a convex domain  in $\R^d$ satisfying $B(\0,1)\subset  \Omega\subset B(\0,d)$.

	Let $\rho=\rho_{\Og}$ denote the metric on $\Og$ defined in \eqref{eqn:def-dubiner}.  Recall that
	$$B_\rho(\x, r):=\{\y\in\Og:\  \rho(\x,\y)\leq r\},\  \  \x\in\Og,\  r>0. $$
Throughout this section, we will use the letters $C_1, C_2,\dots$ to  denote large positive constants depending only on $d$, and letters $c_1, c_2,\dots$ to denote small positive constants depending only on $d$.

	 Using  Theorem \ref{thm:norm control}, we can find  a constant  $C_1>1$  such that  for any $Q\in\Pi_n^d$ and $\x,\y\in\Omega$,
	\begin{equation}\label{7-1}
	|Q(\x)-Q(\y)|\le  C_1 n\rho(\x,\y)\|Q\|_{\Og}.
	\end{equation}
Let $\da:=\f 1 {2C_1}$, and let  $\{\x_j\}_{j=1}^N$ be a maximal $\f \da n$-separated subset of $\Og$ w.r.t. the metric $\rho$; namely,  $\rho(\x_i, \x_j)\ge \f \da n$ for any $1\leq i\neq j\leq N$ and $\Og= \bigcup_{j=1}^N B_\rho(\x_j, \f \da n)$.  Let $\x^\ast\in\Og$ be such that $|Q(\x^\ast)|=\|Q\|_\Og$. Then there exists an integer  $1\leq j_0 \leq N$ such that
 $\x^\ast\in B_\rho(\x_{j_0}, \f \da n)$.
Using \eqref{7-1}, we obtain
\[ \|Q\|_{\Og} -|Q(\x_{j_0})|=|Q(\x^{\ast})|-|Q(\x_{j_0})|\leq C_1\da  \|Q\|_{\Og}=\f 12 \|Q\|_{\Og},\]
implying
\[ \|Q\|_{\Og}\leq 2 |Q(\x_{j_0})|\leq  2 \max_{1\leq j\leq N} |Q(\x_{j})|.\]

Recall that $\dim \Pi_n^d\sim n^d$, where the constants of equivalence depend  only on $d$.
Thus, to  complete the proof of Theorem \ref{thm:optimal mesh}, it is sufficient to show that there exists a positive integer $\ell$ depending only on $d$ such that
$N\leq \dim\Pi_{\ell^2 n}^d$.  Now let $\ell>1$ be  a fixed   large positive integer  depending only on $d$, which will be specified later.
Assume to the contrary that  $N>\dim\Pi_{\ell^2 n}^d$.
We will get  a contradiction as follows. 	By Theorem \ref{thm:polynomial construction}, for each $1\leq j\leq N$, we can find a polynomial $P_j\in\Pi_{\ell^2 n}^d$   such that $P_j(\x_j)=1$ and
	\begin{equation}\label{6-2-1}
	0\leq P_j(\x) \leq C_2 \exp \Bl( -c_2\sqrt{\ell^2 n \rho(\x, \x_j)}\Br),\   \ \forall \x\in\Og.
	\end{equation}
Since  $N>\dim\Pi_{\ell^2 n}^d$,  the polynomials $P_1, \dots, P_N$ are linearly dependent in the space $\Pi_{\ell^2 n}^d$, which means that
	\begin{equation}\label{7-2-0}
	\sum_{j=1}^N a_j P_j(\x)=0,\   \ \forall \x\in\Og,
	\end{equation}
	for some nonzero  vector  $(a_1, \dots, a_N)\in\RR^N\setminus \{\0\}$.
	Without loss of generality, we may assume that
	$1=a_1 =\max_{1\leq j \leq N} |a_j|,$
	and
	\begin{equation}\label{7-2}
	\rho(\x_1, \x_2)\leq \rho(\x_1, \x_3) \leq \rho(\x_1, \x_4)\leq \dots \leq \rho(\x_1, \x_N).
	\end{equation}
	Using  \eqref{7-2-0} and \eqref{6-2-1}, and setting   $t_j:=\rho(\x_j, \x_1)$ for $2\leq j\leq N$,  we obtain    \begin{align}\label{8-1-0}
	1=|P_1(\x_1)|=\Bl|\sum_{j=2}^N a_j P_j(\x_1)\Br|\leq \sum_{j=2}^N |P_j(\x_1)|\leq C_2\sum_{j=2}^N \exp\Bl( -c_2\sqrt{\ell^2 n t_j}\Br).
	\end{align}
	For each  fixed $2\leq j\leq N$, \eqref{7-2} implies that
	$ \{\x_1,\dots, \x_j\}$ is a $\f \da n$-separated subset of $ B_\rho (\x_1, t_j)$, and hence
	 the  balls
	$B_\rho(\x_i, \f \da {2n})$, $i=1,\dots, j$ are pairwise disjoint subsets of  $B_\rho (\x_1, 2t_j)$. It  then follows that
	\begin{equation}\label{7-4}
	\sum_{i=1}^j \lambda_d\left(B_\rho\Bl(\x_i, \f \da {2n}\Br)\right) \leq \lambda_d(B_\rho (\x_1, 2t_j)).
	\end{equation}
	  On the other hand,  however, since for each $1\leq i\leq j$,
	$B_\rho (\x_1, 2t_j)\subset B_\rho (\x_i, 3t_j)$, it follows  by the doubling property stated in  Theorem \ref{thm:doubling} that
	\[ \lambda_d(B_\rho (\x_1, 2t_j))\leq  (12 n t_j/\da)^{2d} \lambda_d\left(B_\rho\Bl(\x_i, \f \da {2n}\Br)\right)=C_3 (n t_j)^{2d}\lambda_d\left(B_\rho\Bl(\x_i, \f \da {2n}\Br)\right) ,\   \ i =1,\dots, j. \]
	This combined with \eqref{7-4} implies
	$j\leq C_4 (nt_j)^{2d}$ for each  $2\leq j\leq N$.
	Thus, using \eqref{8-1-0}, we obtain \begin{align*}
	1\leq &C_2 \sum_{j=2}^N \exp\Bl( -c_3 \ell (j/C_d)^{\f 1 {4d}}\Br)\leq C_5\int_1^\infty \exp\Bl( -c_4 \ell x^{\f1{4d}}\Br)\, dx \\
	&\leq  \f {4dC_5}{ (c_4\ell)^{4d} } \int_{0} ^\infty e^{- u} u^{4d-1}\, du\leq (c_6 \ell)^{-4d}.
	\end{align*}
	This is impossible   if $c_6\ell>1$.    Thus, taking $\ell$ to be the smallest integer $>\f 1 {c_6}$, we obtain
	$N\leq \dim\Pi_{\ell^2 n}^d$,
	 which is as desired.

{\bf Acknowledgment.} The authors thank the referees for the valuable comments and suggestions.

\end{document}